\documentclass[11pt,letterpaper]{amsart} 
\usepackage[
  hmarginratio={1:1},     
  vmarginratio={1:1},     
  textwidth=420pt,        
  heightrounded,          
]{geometry}

\usepackage{amsmath,amssymb,amsfonts,amsthm,amscd,graphicx,psfrag,epsfig,fge}
\usepackage{xcolor}
\usepackage{caption}
\usepackage{tikz-cd}
\usepackage{subcaption}
\definecolor{blue}{rgb}{0,0,0.7}
\definecolor{red}{rgb}{0.75, 0, 0}
\definecolor{midnight}{rgb}{0.0,0.2,0.4}
\usepackage{stmaryrd}
\usepackage{comment}
\usepackage[titletoc,toc]{appendix}
\usepackage{hyperref}
\hypersetup{colorlinks=true, citecolor=blue,linkcolor=midnight}
\usepackage[all]{xy}           
\usepackage{marginnote}  
\usepackage{float}
\usepackage{mathtools}
\usepackage{shuffle}
\usepackage{empheq}
\usepackage{bigdelim}
\usepackage{tikz}
\usepackage{tkz-euclide}
\usetikzlibrary[topaths]
 \usetikzlibrary{arrows}

\newcount\mycount
\usepackage[T2A,T1]{fontenc}
\usepackage[utf8]{inputenc}
\usepackage[russian,english]{babel}
\usepackage{epigraph}

\numberwithin{equation}{section}

\usepackage{amsmath,amssymb,amsfonts,amsthm,amscd,graphicx,psfrag,epsfig,mathabx}
\usepackage{color}
\definecolor{blue}{rgb}{0,0,0.7}
\definecolor{red}{rgb}{0.75, 0, 0}
\usepackage[titletoc,toc]{appendix}

\newtheorem{theorem}{Theorem}[section]
\newtheorem*{theorem*}{Theorem}

\newtheorem{lemma}[theorem]{Lemma}
\newtheorem{proposition}[theorem]{Proposition}
\newtheorem{corollary}[theorem]{Corollary}
\newtheorem{conjecture}[theorem]{Conjecture}

\newtheorem{remark}[theorem]{Remark}

\newcommand{\be}{\begin{equation}}
\newcommand{\bs}{\begin{split}}
\newcommand{\ee}{\end{equation}}
\newcommand{\es}{\end{split}}
\newcommand{\lra}{\longrightarrow}

\def\F{\mathrm{F}}

\def\P{\mathrm{P}}

\def\L{\mathcal{L}}
\def\C{\mathbb{C}}
\def\Q{\mathbb{Q}}

\def\Z{\mathbb{Z}}

\def\Li{\textup{Li}}

\DeclareMathOperator{\cyc}{cyc}

\newcommand\rsmraise[1]{%
  \ifx#1\displaystyle .8\else
    \ifx#1\textstyle .8\else
      \ifx#1\scriptstyle .6\else
        .45%
      \fi
    \fi
  \fi}
 
\makeatletter 
\newcommand\nobreakpar{\par\nobreak\@afterheading} 
\makeatother

\date{}
\title[]{Cluster Polylogarithms I:\\ Quadrangular Polylogarithms}
\author[]{Andrei Matveiakin}
\author[]{Daniil Rudenko}

\begin{document}

\begin{abstract} We suggest a definition of cluster polylogarithms on an arbitrary cluster variety and classify them in type $A$. We find functional equations for multiple polylogarithms which generalize equations discovered by Abel, Kummer, and Goncharov to an arbitrary weight. As an application, we prove a part of the Goncharov depth conjecture in weight six. 
\end{abstract}

\maketitle
\pagebreak 
\tableofcontents

\pagebreak 
\section{Introduction}
\subsection{Polylogarithms and mixed Tate motives} \label{SectionIntroductionPolylogarithms}
Multiple polylogarithms are certain multivalued functions of variables
$
a_1,\dots,a_k\in \C
$ 
depending on integer parameters $n_1,\dots,n_k\in \mathbb{N}.$ In the polydisc   $|a_1|,|a_2|,\dots, |a_k| <1$ a multiple polylogarithm can be defined by the power series
\[
\Li_{n_1,n_2,\dots, n_k}(a_1,a_2,\dots,a_k)=\sum_{m_1>m_2>\dots>m_k>0}\frac{a_1^{m_1} a_2^{m_2}\dots a_k^{m_k}}{m_1^{n_1}m_2^{n_2}\dots m_k^{n_k}}.
\]
The number $n=n_1+\dots+n_k$ is called the weight of the multiple polylogarithm, and the number $k$ is called its depth. In this form, Goncharov defined these functions in \cite{Gon95B}, but they appeared in different disguises long before that.

An explanation of the properties of multiple polylogarithms is provided by the theory of mixed Tate motives (see \cite{Gon95}); here is a brief summary. One defines a graded Lie coalgebra $\mathcal{L}_\bullet(\mathrm{F})$ of multiple polylogarithms with values in a field $\mathrm{F}$. It is generated by symbols $\textup{Li}_{n_0;n_1,n_2,\dots, n_k}^{\L}(a_1,a_2,\dots,a_k)$ for $a_1,\dots,a_k\in \mathrm{F}^\times$ which are subject to (unknown) functional equations for polylogarithms. The coproduct $\Delta \colon \mathcal{L}(\mathrm{F}) \lra \Lambda^2 \mathcal{L}(\mathrm{F})$ was defined by Goncharov based on the properties of the mixed Hodge structure related to polylogarithms. Lie coalgebra $\mathcal{L}(\F)$ is expected to coincide with the fundamental Lie coalgebra of the category of mixed Tate motives over $\F.$ Guided by that, one can explain the relation between polylogarithms and volumes of hyperbolic polytopes, special values of zeta functions, algebraic K-theory, and motivic cohomology. 

The conjectures of Goncharov predict  the existence of various formulas involving polylogarithms: functional equations, formulas for Chern classes, and depth reduction formulas. Nevertheless, they do not hint at how to find such formulas  explicitly. Goncharov envisioned that the explicit structure of the formulas should be related to {\bf cluster structures} on algebraic varieties. The Abel five-term relation gives the simplest example of this phenomena: the arguments of dilogarithms are the five cross-ratios, which are the cluster $\mathcal{X}$-coordinates on cluster variety $\mathfrak{M}_{0,5}.$ 

We suggest a definition of cluster polylogarithms on a cluster variety, see \S  \ref{SectionDefinitionClusterPolylogarithms}. The definition we give has been in the air for a long time; the key part of it is {\it the cluster adjacency property} for integrable symbols,  which appeared in physics literature (\cite{DFG18}). The existence of cluster polylogarithms of depth greater than one is far from obvious. We discovered the concept in an attempt to explain the properties of quadrangular polylogarithms, which played a key role in \cite{Rud20}. 

In \S \ref{SectionClusterConfigurationSpace}  we prove that  quadrangular polylogarithms span the space of cluster polylogarithms on $\mathfrak{M}_{0,m}$ and find a unique functional equation they satisfy. This equation implies various known functional equations found by Abel, Kummer, Goncharov, and Gangl. This is strong evidence that the defining equations for the Lie coalgebra of mixed Tate motives have cluster origin. This might be a shadow of certain non-commutative ``quantum mixed Tate motives'' yet to be discovered.

One of the most fundamental questions about polylogarithms is the nature and properties of the filtration by depth. Goncharov suggested an ambitious conjecture, giving a necessary and sufficient condition for a sum of polylogarithms to have a certain depth. In \S \ref{SectionHigherGangl} we show that our results can be used to prove a part of the depth conjecture in weight six.

\subsection{Cluster polylogarithms and iterated integrals}  \label{ClusterPolylogarithms}
There have been a few instances when the connection between cluster structures and polylogarithms manifested itself. Fock and Goncharov discovered that the dilogarithm appears as the generating function of a cluster mutation, see  \cite{FG09}. Next, the relation between scattering amplitudes, polylogarithms, and cluster varieties was noticed in \cite{GSVV10} and \cite{GGSVV14}. In particular, in \cite[Appendix B]{GGSVV14} a $40$-term equation for $\Li_3$ was found, where all arguments are cluster $\mathcal{X}$-coordinates on $\textup{Gr}(3,6).$ Finally, the functional equation relating  $\Li_{3,1}$ and  $\Li_4$ in \cite{GR18} was obtained by integrating the exponent of the $K_2$-symplectic form on $\mathfrak{M}_{0,7},$ so is of cluster origin as well.  In this section, we give an informal definition of cluster polylogarithms; a precise definition is given in \S \ref{SectionDefinitionClusterPolylogarithms}.

Cluster algebras were introduced by Fomin and Zelevinsky in \cite{FZ02}; a different approach to the subject was developed by Fock and Goncharov in \cite{FG06},  \cite{FG09}.  A cluster algebra  is a commutative ring equipped with a distinguished set of generators (called cluster variables or cluster $\mathcal{A}$-coordinates)  grouped into overlapping subsets called clusters. The clusters have the following property: for any cluster ${\bf a}$ and cluster variable $a\in {\bf a}$ there exists another cluster obtained from   ${\bf a}$ by replacing $a$ with another variable $a'$ related to $a$ by an exchange relation
\be \label{FormulaClusterExchange}
a a'=M_1+M_2
\ee
for monomials $M_1$ and $M_2$ in variables of the cluster ${\bf a}$ different from $a.$ The operation of passing from cluster ${\bf a}$ to $({\bf a}- \{a\})\cup \{a'\}$  is called a mutation in $a$; by a sequence of such operations, we obtain all clusters in the cluster algebra. 

Consider an irreducible affine algebraic variety $X$ over $\mathbb{Q};$ let $\mathbb{Q}[X]$ be the corresponding algebra of regular functions. Assume that  the ring $\mathbb{Q}[X]$ has a structure of cluster algebra. Every cluster variable $a\in \mathbb{Q}[X]$ defines a differential $1$-form $d\log(a)=\dfrac{d a}{a}\in \Omega^1_X.$
Every exchange relation (\ref{FormulaClusterExchange}) defines an identity in $\Omega^2_X$:
\be \label{FormulaClusterRelation}
d\log(a a')\wedge d\log(M_1)+d\log(M_2)\wedge d\log(a a') +d\log(M_1)\wedge d\log(M_2)=0.
\ee

Next, recall the notion of a  Chen iterated integral on the complex manifold $X^{sm}(\mathbb{C}).$ Let $\omega_1,\dots,\omega_n$ be  $1-$forms on $X$ and $\gamma\colon [0,1]\lra X^{sm}(\mathbb{C})$ be a piecewise smooth path. Consider the pullbacks of  forms $\gamma^*(\omega_i)=f_i(t)dt.$ The iterated integral 
is defined in the following way:
\[
\int_{\gamma} \omega_1\circ\dots \circ \omega_n=\int_{0\leq t_1\leq \dots\leq t_n\leq 1}f_1(t_1)dt_1\dots f_n(t_n)dt_n.
\]
Extending by linearity, we obtain a map
$
\int_\gamma \colon \bigl(\Omega^1_X\bigr)^{\otimes n}\lra \mathbb{C};
$
we will denote  $\omega_1\otimes\dots\otimes\omega_n$ by  $[\omega_1|\dots|\omega_n].$ In general, an iterated integral depends on the path $\gamma.$ Chen proved that it depends only on the homotopy class of the path if and only if 
\be\label{FormulaHomotopyInvariance}
\sum_{i=1}^n [\omega_1|\dots|\omega_{i-1}|d\omega_{i}|\omega_{i+1}|\dots|\omega_n ]+\sum_{i=1}^{n-1} [\omega_1|\dots|\omega_{i-1}|\omega_{i}\wedge\omega_{i+1}|\dots|\omega_n ]=0.
\ee

Informally, a cluster polylogarithm is a homotopy-invariant iterated integral 
\[
\int_\gamma \sum_{i} \bigl[d \log(a_1^i)|\dots|d \log(a_n^i)\bigr]
\]
on $X^{sm}(\mathbb{C})$ where for each $i$ there exists a cluster containing cluster variables $a_1^i,\dots,a_n^i.$ We call the latter condition {\it cluster adjacency}, it was inspired by \cite{DFG18}.

Consider the following simplest example. Let $X$ be a cone over Grassmannian $\textup{Gr(2,4)};$ denote by $\Delta_{ij}$ for $0\leq i<j\leq 3$ the Pl{\"u}cker coordinates. Then $X$ is a hypersurface in $\mathbb{C}^6$ defined by an equation
\be \label{FormulaPlucker}
\Delta_{02}\Delta_{13}=\Delta_{01}\Delta_{23}+\Delta_{03}\Delta_{12}.
\ee
The corresponding cluster algebra is said to have type $A_1;$ it has two clusters:
\[
\{\Delta_{01},\Delta_{12},\Delta_{23},\Delta_{03},\Delta_{02}\}
\text{\ \ and \ \ }
\{\Delta_{01},\Delta_{12},\Delta_{23},\Delta_{03},\Delta_{13}\}.
\]
The only exchange relation is (\ref{FormulaPlucker}). Iterated integral
\be\label{FormulaDilogarithm}
\int_\gamma \left[ d\log \left(-\frac{\Delta_{01}\Delta_{23}}{\Delta_{03}\Delta_{12}}\right)\bigg| d\log \left(-\frac{\Delta_{02}\Delta_{13}}{\Delta_{03}\Delta_{12}}\right)\right],
\ee
 is homotopy invariant. It satisfies the cluster adjacency condition, because cluster variables $\Delta_{01}, \Delta_{12},\Delta_{23},\Delta_{03}$ lie in both clusters and we have:
\begin{align*}
\left[ d\log \left(-\frac{\Delta_{01}\Delta_{23}}{\Delta_{03}\Delta_{12}}\right)\bigg| d\log \left(-\frac{\Delta_{02}\Delta_{13}}{\Delta_{03}\Delta_{12}}\right)\right]=&-\left[ d\log \left(-\frac{\Delta_{01}\Delta_{23}}{\Delta_{03}\Delta_{12}}\right)\bigg| d\log \left(-\Delta_{03}\Delta_{12}\right)\right]\\
&+\left[ d\log \left(-\frac{\Delta_{01}\Delta_{23}}{\Delta_{03}\Delta_{12}}\right)\bigg|  d\log (\Delta_{02})\right]\\
&+\left[ d\log \left(-\frac{\Delta_{01}\Delta_{23}}{\Delta_{03}\Delta_{12}}\right)\bigg| d\log (\Delta_{13})\right].
\end{align*}
For a certain $\gamma$, (\ref{FormulaDilogarithm}) is equal to  the dilogarithm of the cross-ratio
$\Li_2 \left(-\dfrac{\Delta_{01}\Delta_{23}}{\Delta_{03}\Delta_{12}}\right).$

In \S \ref{SectionDefinitionClusterPolylogarithms} we associate to every cluster algebra $\mathbb{Q}[X]$ a Lie coalgebra of {\it cluster integrable symbols} $\textup{CL}_\bullet(X)$ from which cluster polylogarithms are obtained by integration. In all examples that we are aware of, there exists a canonical way to integrate an element of $\textup{CL}_\bullet(X)$ to a multivalued function. 

\subsection{Quadrangular polylogarithms are cluster polylogarithms in type $A$} \label{SectionIntroductionClusterPolylogsOnConfigurationSpace}

Examples of cluster  polylogarithms are not easy to construct because the conditions of cluster adjacency and homotopy invariance seem to be ``transversal'' to each other. In this section, we classify cluster integrable symbols on the cone over the Grassmannian $\textup{Gr}(2,m+2),$ which has  cluster algebra structure of  type~$A_{m-1}$. We will see that the corresponding cluster polylogarithms are precisely the quadrangular polylogarithms introduced in \cite{Rud20}.

Consider a Grassmannian $\textup{Gr}(2,m+2)$ in its Pl{\"u}cker embedding and denote by $\Delta_{ij}$ for $0\leq i<j\leq m+1$  the Pl{\"u}cker coordinates. In order to define  cluster algebra structure on $\mathbb{Q}[\textup{Gr}(2,m+2)]$, consider a convex polygon with vertices  labeled by indices $0,\dots, m+1.$ The cluster variables are Pl{\"u}cker coordinates $\Delta_{ij};$ they  correspond to sides and diagonals of the convex $(m+2)$-gon. Two cluster variables lie in the same cluster if and only if the corresponding chords have no common interior points. We will see that cluster integrable symbols of weight greater than one are invariant under the torus action on the Grassmannian and can be viewed as functions on $\mathfrak{M}_{0,m+2}.$ We denote the corresponding  space of cluster integrable symbols by $\textup{CL}_\bullet\left(\mathfrak{M}_{0,m+2}\right).$

For $m=4,5,$ the only cluster integrable symbols of weight $n\geq 2$ are symbols of classical polylogarithms of cross-ratios. For $m=6,$ a new cluster integrable symbol appears in weight four. This function (modulo lower depth corrections) is known in physics literature under the name  ``$A_3$-function'', in \cite{GR18} under the name $\mathrm{L}_{4}^1$, and  in \cite{Rud20} under the name ``quadrangular polylogarithm $\textup{QLi}_{4}.$'' 

\begin{theorem}\label{TheoremClusterTypeA} Dimension of the space  $\textup{CL}_n\left (\mathfrak{M}_{0,m+2}\right)$ of weight $n\geq 2$ equals to
\[
{m+1 \choose 3}+{m+1 \choose 4}+\dots+{m+1 \choose n+1}.
\] 
The space $\textup{CL}_n\left (\mathfrak{M}_{0,m+2}\right)$ is generated by symbols of quadrangular polylogarithms 
\[
\textup{QLi}_{n}(x_{i_0},\dots,x_{i_{2r+1}})
\quad \text{ for }\quad  0\leq i_0< \dots< i_{2r+1}\leq m+1.
\]
\end{theorem}

Quadrangular polylogarithms satisfy the following equation: for $N \geq n+2$, we have
\begin{equation}\label{FormulaEquationQuadrangular}
\sum_{0\leq i_0<\dots <i_{2r+1}\leq N}(-1)^{i_0+\dots+i_{2r+1}}\textup{QLi}_{n}^{\textup{sym}}(x_{i_0},x_{i_1},\dots, x_{i_{2r+1}})=0,
\end{equation}
where $\textup{QLi}_{n}^{\textup{sym}}$ is a symmetrized version of $\textup{QLi}$ defined in \S \ref{SectionQuadrangular}. In \S \ref{SectionSpaceClusterIntegrable} we show that all linear relations between quadrangular polylogarithms follow from  (\ref{FormulaEquationQuadrangular}).

There exist two presentations for quadrangular polylogarithms. The first one expresses quadrangular polylogarithms via correlators (\cite[Definition 5.2]{Rud20}), and the second one, called {\it quadrangulation formula}, expresses them via multiple polylogarithms (\cite[Theorem 1.2]{Rud20}).   Equation (\ref{FormulaEquationQuadrangular}) can be combined with the quadrangulation formula to give new equations for polylogarithms.  Taking $n=2, N=5,$ we get the Abel $5$-term relation for $\Li_2$;   for $n=3, N=6,$ we get the 9-term relation of Kummer and  the 22-term relation of Goncharov for $\Li_3$, and for $n=4, N=7,$  we get the equation
${\bf \mathbf{Q}_{4}}$ from \cite[\S1.2.1]{GR18} and  the  931-term relation for $\Li_4$ found by Gangl (see \cite{Gan16}). We are tempted to conjecture that all functional equations for multiple polylogarithms follow from (\ref{FormulaEquationQuadrangular}).

We expect that analogs of Theorem  \ref{TheoremClusterTypeA} hold in much greater generality. In particular, we have experimental evidence that similar results hold for all cluster algebras of finite type. 

The original construction of quadrangular polylogarithms appeared rather ad hoc. In \S \ref{SectionSpaceClusterIntegrable} we prove Corollary \ref{CorollaryProjectiveInvariantII}, which gives yet another explanation of the role played by quadrangular polylogarithms. Notice that iterated integrals 
$
\textup{I}(x_0;x_1,\dots,x_n;x_{n+1})
$
are invariant under affine transformations but not under projective transformations.  Apparently, it is possible to ``correct'' them by adding degenerate terms of the form 
\[
\textup{I}(x_{i_0};x_{i_1},\dots,x_{i_n};x_{i_{n+1}}) \quad \text{for}\quad 0\leq i_0 \leq \dots\leq i_{n+1}\leq n+1
\]
to restore  projective invariance. The resulting ``corrected'' functions are precisely quadrangular polylogarithms.

\subsection{Application: higher Gangl formula in weight six} \label{SectionGoncharovDepth}
 Consider a Lie coalgebra $\mathcal{L}(\mathrm{F})$ discussed in \S \ref{SectionIntroductionPolylogarithms}. This coalgebra is filtered by depth; denote by $\textup{gr}^{\mathcal{D}}_\bullet \mathcal{L}(\mathrm{F})$ the associated graded space. The subspace spanned by classical polylogarithms $\Li_n(a)$ is denoted by $\textup{gr}^{\mathcal{D}}_1\mathcal{L}(\mathrm{F})=\mathcal{B}_n(\mathrm{F}).$  Now we are ready to formulate the depth conjecture of Goncharov (\cite[Conjecture 7.6]{Gon01}).

\begin{conjecture}[Depth conjecture] \label{ConjectureDepth}	A linear combination of multiple polylogarithms has depth less than or equal to $k$ if and only if its $k-$th iterated truncated coproduct vanishes. Moreover,  the Lie coalgebra $\textup{gr}^{\mathcal{D}}_\bullet \mathcal{L}(\mathrm{F})$ is cofree, cogenerated by polylogarithms of depth one.
\end{conjecture}

The depth conjecture  has several remarkable consequences. For instance, it would imply that volumes of hyperbolic {\it manifolds} can be expressed via classical polylogarithms. The main reason is that coproduct of the volume of a polytope is related to the Dehn invariant (see \cite[\S 3]{Gon99}), and the Dehn invariant of a hyperbolic manifold is equal to zero. Also, the depth conjecture would be a crucial part of the proof of  Zagier conjecture; the remaining part would be to show that special values of zeta function can be expressed via multiple polylogarithms.  

In weights two and three, the depth conjecture was known to be true for a long time. In weight four, it was proved by Gangl (see \cite{Gan16}) after many months of computer-assisted search. Another proof exploiting the geometry of cluster varieties  was given in (\cite[Theorem 1.14]{GR18}).
The ``unobstructed'' case $2k\geq n$ was proved in (\cite[Theorem 1.1]{Rud20}) using the properties of quadrangular polylogarithms. Also, see \cite{Cha17} and \cite{CGR19b} for some other results about depth reduction.

For $n=2k+2$ Conjecture \ref{ConjectureDepth} reduces to the following explicit statement:
\begin{conjecture}[Depth conjecture, first obstructed case] \label{ConjectureDepthObstr} For $a_1,\dots,a_{k-1}\in \F^{\times}$ and $x_0,\dots,x_4\in \mathbb{P}^1_\F$  the sum
\[ 
	\sum_{i=0}^4(-1)^i\textup{Li}_{k;1,\dots,1}^{\L}(a_1,\dots,a_{k-1},[x_0,\dots,\widehat{x_i},\dots,x_4]) 
\]
has depth at most $k-1.$
\end{conjecture}

All known cases of the depth conjecture can be derived from a combination of equation (\ref{FormulaEquationQuadrangular}) and the quadrangulation formula. We prove that equation (\ref{FormulaEquationQuadrangular}) for $n=6, N=9$ implies the following result, analogous to Gangl's formula in weight six.

\begin{theorem}\label{TheoremDepth6} For $a_1, a_2\in \F^{\times}$ and $x_0,\dots,x_4\in \mathbb{P}^1_\F$  the sum
\[
\sum_{i=0}^4(-1)^i\textup{Li}_{3;1,1,1}^{\L}(a_1,a_2,[x_0,\dots,\widehat{x_i},\dots,x_4]) 
\] can be expressed via polylogarithms of depth at most two and a function
\be\label{FormulaZagierWeightSix}
\Li_{3;1,1,1}^{\L}(b_1,b_2,b_3)+\Li_{3;1,1,1}^{\L}(b_1,b_2,1-b_3)\quad \text{for}\quad b_1, b_2, b_3\in \F^{\times}.
\ee
\end{theorem}

\paragraph{\bf{Acknowledgements}} 
We thank A.~Beilinson, S.~Bloch, S.~Charlton,  V.~Fock,  S.~Fomin, B.~Farb,  H.~Gangl,  A.~Goncharov, D.~Krachun, and D.~Radchenko for useful discussions and comments.

\section{Lie coalgebra of multiple polylogarithms}  \label{SectionLieCoalgebraMP} 
\subsection{Inductive definition of the Lie coalgebra of multiple polylogarithms}
Let $\F$ be a field of characteristic zero. We aim to define the graded Lie coalgebra $\L_\bullet(\F)$ of multiple polylogarithms. This definition mimics Goncharov's construction of higher Bloch groups in \cite{Gon95B}. A similar construction appeared  in \cite{GoncharovMSRI}, see also  \cite{GKLZ}.

Consider a $\Q-$vector space  $\mathcal{A}_n(\F)$ for $n\geq 1$ generated by ordered tuples of points $(x_0,x_1,\dots,x_n)\in \F^{n+1}$ modulo relations
\[
(x_0,x_1,\dots, x_{n-1}, x_n)=(x_1,x_2,\dots,x_n,x_0)
\] 
and $(x,x,\dots,x)=0.$ We define a Lie cobracket 
$\Delta\colon \mathcal{A}_\bullet(\F)\lra \Lambda^2 \mathcal{A}_\bullet(\F)$
by formula
\[
\Delta (x_0,\dots, x_n)=\sum_{\cyc}\sum_{i=1}^{n-1}(x_{0}, x_1, \dots, x_i)\wedge (x_{0},  x_{i+1}, \dots, x_n),
\]
where we put 
\[
\sum_{\cyc}f(x_0,\dots,x_n)=\sum_{j\in \mathbb{Z}/(n+1)\mathbb{Z}}f(x_j,\dots,x_{j+n}).
\]
The coJacobi identity can be easily verified.

Consider a field $\mathrm{K}$ with a discrete valuation $\nu\colon \mathrm{K}\lra \Z\cup\{\infty\};$ denote by $\mathrm{k}$  the residue field. Given a  uniformizer $\pi\in \mathrm{K},$  we define specialization homomorphism 
\[
\mathrm{Sp}_{\nu, u}\colon \mathcal{A}_n(\mathrm{K})\lra \mathcal{A}_n(\mathrm{k})
\] 
by the following formula. Let $m=\textup{min}_{0\leq j \leq n} \ \nu(x_j).$ Then, we have
\be \label{FormulaSpecialization}
\mathrm{Sp}_{\nu, u}(x_0,x_1,\dots, x_n)=(y_0,y_1,\dots,y_n),
\ee
where
\[
y_i=
\begin{cases}
\overline{x_i \pi^{-m}} & \text{\ if \ }\nu(x_i)=m,\\
0 & \text{\ if \ }\nu(x_i)>m.\\
\end{cases}
\]

\begin{lemma}\label{LemmaSpecializationCoproduct}
For any discrete valuation on $\F$ the Lie cobracket $\Delta\colon \mathcal{A}_\bullet(\F)\lra \Lambda^2 \mathcal{A}_\bullet(\F)$ commutes with the specialization.
\end{lemma}
\begin{proof}
We compare $\Delta\Bigl(\textup{Sp}(x_0,\dots,x_n)\Bigr)$
and 
\be\label{FormulaSpCopr}
\sum_{\cyc}\sum_{i=1}^{n-1}\textup{Sp}(x_{0}, x_1, \dots, x_i)\wedge \textup{Sp}(x_{0}, x_{i+1}, \dots, x_n).
\ee
term by term. The only case when the corresponding terms are distinct is when 
\[
\nu(x_{j}), \nu(x_{j+1}), \dots,\nu(x_{j+i})>m
\]
or if 
\[
\nu(x_j), \nu(x_{j+i+1}), \dots,\nu(x_{j+n})>m
\]
for some $j.$ Indeed, in this case the corresponding term in $\Delta\left(\textup{Sp}(x_0,\dots,x_n)\right)$ vanishes, while
 \[
 \textup{Sp}(x_j, x_{j+1},\dots,x_{j+i})\wedge \textup{Sp} (x_j, x_{j+i+1}\dots,x_{j+n})
 \] may not.

Assume without loss of generality that $
\nu(x_0), \nu(x_1), \dots,\nu(x_{i})>m.$ The term 
\[
\textup{Sp}(x_{0}, x_1,\dots, x_i)
\]
 appears exactly twice in (\ref{FormulaSpCopr}): in 
\be \label{FormulaTermOne}
\textup{Sp}(x_{0}, x_1, \dots, x_i)\wedge \textup{Sp}(x_0, x_{i+1}, \dots, x_{n}).
\ee
and in 
\be\label{FormulaTermTwo}
\textup{Sp}(x_{i},  x_{i+1},\dots, x_{n})\wedge \textup{Sp}(x_i, x_{0},  \dots, x_{i-1}).
\ee
Since by our assumption $\nu(x_0)>m$ and $\nu(x_i)>m,$
\[
\textup{Sp}(x_0, x_{i+1}, \dots, x_{n})=\textup{Sp}(x_{i},  x_{i+1},\dots, x_{n}),
\]
and so terms (\ref{FormulaTermOne}) and (\ref{FormulaTermTwo}) cancel out in (\ref{FormulaSpCopr}). This finishes the proof of the lemma.
\end{proof}

We define the space of relations $\mathcal{R}_n(\F)\subseteq \mathcal{A}_n(\F)$ inductively; the coalgebra of multiple polylogarithms is defined as a quotient
\be \label{FormulaQuotient}
\L_n(\F)=\frac{\mathcal{A}_n(\F)}{\mathcal{R}_n(\F)}.
\ee
The projection of $(x_0,x_1,\dots, x_n)\in \mathcal{A}_n(\F)$ to $\L_n(\F)$ is denoted by $\textup{Cor}^{\L}(x_0,x_1,\dots, x_n)$ and called {\it the correlator}. 

In weight one we define $\mathcal{R}_1(\F)$ to be the kernel of the map sending $(x_0, x_1)\in \mathcal{A}_1(\F)$ to $(x_0-x_1)\in \F^\times_{\Q}.$ By definition, $\L_1(\F)\cong \F^\times_\Q.$ We will also denote $\textup{Cor}^{\L}(0,a)$ by $\log^{\L}(a).$ 
	
Next, assume that spaces $\mathcal{R}_i(\F)$ are defined in weights less than $n.$ Consider the field $\mathrm{K}=\F(t);$ let $\nu_a$ be a discrete valuation corresponding to $a\in \mathbb{P}^1_{\F}.$ In this setting, we denote the specialization $\textup{Sp}_{\nu_a, t-a}$ by $\textup{Sp}_{t\to a}$ and $\textup{Sp}_{\nu_{\infty}, \frac{1}{t}}$ by $\textup{Sp}_{t\to \infty}.$ The space $\mathcal{R}_n(\F)$ is spanned by elements
\[
\mathrm{Sp}_{t \to 0}(R(t))-\mathrm{Sp}_{t \to \infty}(R(t))\in \mathcal{A}_n(\F)
\]
for elements $R(t) \in \mathcal{A}_n(\F(t))$ with zero cobracket in $\Lambda^2 \L_\bullet(\F(t)).$  

We define the weight $n$ component of the Lie coalgebra of polylogarithms by formula (\ref{FormulaQuotient}). By Lemma~\ref{LemmaSpecializationCoproduct}, we have 
\[
\Delta\Bigl(\mathrm{Sp}_{t \to 0}(R(t))-\mathrm{Sp}_{t \to \infty}(R(t))\Bigr)=0,
\]
so the coproduct $\Delta$ descends to $\L(\F).$ We have
\be\label{FormulaCoproductCorrelators}
\Delta \textup{Cor}^{\L}(x_0,\dots, x_n)=\sum_{\cyc}\sum_{i=1}^{n-1} \textup{Cor}^{\L}(x_{0}, x_1, \dots, x_i)\wedge \textup{Cor}^{\L}(x_{0},  x_{i+1}, \dots, x_n).
\ee
This finishes an inductive definition of the Lie coalgebra of multiple polylogarithms.

\begin{remark}
By Lemma~\ref{LemmaSpecializationCoproduct}, specialization homomorphisms are well-defined on $\L_\bullet(\F).$ It is easy to see that for 
$a = u\pi^m\in \L_1(\mathrm{K})$ with $\nu(u)=0$ we have 
$S_{\nu, \pi} (a)=\overline{u}\in \mathrm{k}^\times,$
so in weight one, specialization depends on the choice of a uniformizer. In higher weights, specialization does not depend on the choice of a uniformizer (which follows from the affine invariance of correlators of weight at least two).
\end{remark}

\begin{remark} If $\F$ is a subfield of $\C,$ there exists a realization map from $\L_\bullet(\F)$ two the Lie coalgebra of framed Hodge-Tate structures. This map sends correlators to Hodge correlators, defined in  \cite{Gon19}. Moreover, it sends elements 
\be \label{FormulaCorII}
\textup{I}^{\L}(x_0;x_1,\dots,x_n;x_{n+1})=\textup{Cor}^{\L}(x_1,\dots,x_{n+1})-\textup{Cor}^{\L}(x_0,\dots,x_{n})\in \L_n(\F),
\ee
which we call {\it iterated integrals}, to Hodge iterated integrals.
If $\F$ is a number field, there exists a motivic realization as well. 
It is conjectured to be an isomorphism.
\end{remark}

\subsection{Multiple polylogarithms and the depth filtration} 
For an integer $n_0\geq -1,$ positive integers $n_1,\dots,n_k,$ and elements $a_1,\dots,a_k\in \F^\times$ we define   multiple polylogarithm
\[
\begin{split}
	&\Li_{n_0;n_1,\dots,n_k}^{\L}(a_1,a_2,\dots,a_k)\\
&=(-1)^{k}\textup{I}^{\L}(\underbrace{0;0,\dots,0}_{n_0+1},1,\underbrace{0,\dots,0,a_1}_{n_1},\dots,\underbrace{0,\dots,0,a_1a_2\dots a_{k-1}}_{n_{k-1}},\underbrace{0,\dots,0;a_1a_2\dots a_{k}}_{n_k}).
\end{split}
\]
This is an element of $\L_n(\F)$ for $n=n_0+n_1+n_2+\dots+n_k.$ The number  $k$ is called the depth of the multiple polylogarithm. It is often convenient to consider the sum 
\[
\Li_{\bullet;n_1,\dots,n_k}^{\L}(a_1,a_2,\dots,a_k)=\sum_{n_0=0}^{\infty} \Li_{n_0;n_1,\dots,n_k}^{\L}(a_1,a_2,\dots,a_k)\in \L(\F).
\]
If $n_0=0,$ we omit it from the notation: 
\[
\Li_{n_1,\dots,n_k}^{\L}(a_1,a_2,\dots,a_k):=\Li_{0;n_1,\dots,n_k}^{\L}(a_1,a_2,\dots,a_k).
\]

Lie coalgebra $\L_\bullet(\F)$ is filtered by depth: we denote by $\mathcal{D}_k\L_\bullet(\F)$ the subspace spanned by polylogarithms of depth not greater than $k.$ Denote by $\mathrm{gr}^\mathcal{D}_k\L$ the associated graded space. The space  $\mathcal{D}_1\L_n(\F)$ is usually called the higher Bloch group $\mathcal{B}_n(\F),$ it is spanned by classical polylogarithms $\Li_n^{\L}(a)$ for $a\in \F^\times.$ 

In weight one, we have 
\[
\textup{Li}_1^{\L}(a)=-\log^{\L}(1-a)\in \L_1(\F).
\]

\begin{proposition}\label{PropositionWeightTwo} The space $\L_2(\F)$ is spanned by dilogarithms $\Li_2^{\L}(a).$ All relations between them follow from the $5-$term relation  
\[
\sum_{i=0}^4 (-1)^i\Li_2^{\L}([x_0,\dots,\widehat{x_i},\dots,x_4])=0
\]
for distinct $x_0,\dots, x_4\in \mathbb{P}^1_\F.$
\end{proposition}
\begin{proof}
Classical dilogarithms span $\L_2(\F)$ because it is easy to see that
\[
\textup{Cor}^{\L}(x_0,x_1,x_2)=\Li_2^{\L}\left( [\infty,x_0,x_1,x_2]\right)=\Li_2^{\L}\left( \frac{x_1-x_2}{x_1-x_0}\right).
\]
 The fact that all relations follow from the five-term relation is equivalent to the fact that the kernel of the coproduct 
\[
\textup{K}(\F)=\textup{Ker}\left(\Delta \colon \L_2(\F)\lra \Lambda^2\F_\Q^\times \right)
\]
satisfies the homotopy invariance property, i.e., the embedding $\textup{K}(\F)\hookrightarrow \textup{K}(\F(t))$ is an isomorphism. This follows from the results of Suslin, see \cite{Sus90}.
\end{proof}
Now we discuss the behavior of multiple polylogarithms under specialization.
\begin{lemma}
The specialization of $\Li_{n_0;n_1,\dots,n_k}^{\L}(a_1,a_2,\dots,a_k)$
has depth less then or equal to~$k.$ If not all valuations $\nu(a_i)$ are equal to zero, then it has a depth less than~$k.$
\end{lemma}
\begin{proof}
The statement follows immediately from the definition of specialization.
\end{proof}

Specializing the $5-$term relation to various divisors $x_i=x_j$, we obtain
\[
\Li_2^{\L}\bigl([x_{\sigma(1)},x_{\sigma(2)},x_{\sigma(3)},x_{\sigma(4)}]\bigr)=(-1)^{\sigma}\Li_2^{\L}\bigl([x_1,x_2,x_3,x_4]\bigr) \text{\ for \ }\sigma \in \mathbb{S}_4.
\]
In particular, $\Li_2^{\L}(1-a)=-\Li_2^{\L}(a)$ and $\Li_2^{\L}\left(\dfrac{1}{a}\right)=-\Li_2^{\L}(a).$

\subsection{Definition of cluster polylogarithms}  \label{SectionDefinitionClusterPolylogarithms}
Now we are ready to give a formal definition of cluster polylogarithms on a cluster variety $X$.  First, we define the Lie coalgebra of cluster integrable symbols $\textup{CL}_\bullet(X).$ Denote by $A$ a subspace of the rationalization of the group $\mathbb{Q}(X)^{\times}$ generated by cluster variables. The tensor algebra $\mathrm{T}^\bullet(A)$ has a Hopf algebra structure with shuffle product $\shuffle$ and deconcatenation coproduct. Its Lie coalgebra of indecomposable elements is a cofree Lie coalgebra cogenerated by $A;$ we denote it  by $\textup{CoLie}_\bullet(A).$ It is easy to see that $\textup{CoLie}_1(A)\cong A$ and $\textup{CoLie}_2(A)\cong A\wedge A.$  Lie coalgebra   $\textup{CL}_\bullet(X)$ consists of elements in $\textup{CoLie}_\bullet(A)$ satisfying two conditions: {\it cluster adjacency} and {\it integrability}. The latter is closely related to homotopy invariance (\ref{FormulaHomotopyInvariance}).

In weight one we put $\textup{CL}_1(X)=A.$ Next,  $\textup{CL}_2(X)$ is a subgroup of $\textup{CoLie}_2(A)$ generated by elements  $\dfrac{M_1}{aa'} \wedge \dfrac{M_2}{aa'}$ for every exchange relation (\ref{FormulaClusterExchange}).
The Lie cobracket
$
\Delta \colon \textup{CL}_2(X) \lra \Lambda^2\textup{CL}_1(X)
$
is the embedding. An element 
\[
L=\sum_{I=(i_1,\dots,i_n)} n_I [a_{i_1}|a_{i_2}| \dots |a_{i_n}]\in \mathrm{T}^n(A)
\]
is called {\it cluster adjacent} if for every $I$  there exists a cluster containing cluster  variables $a_{i_1}, \dots,  a_{i_n}.$ Next, $L$ is called {\it integrable} if for every $1\leq s \leq n-1 $ the element 
\[
\pi_s(L)= \sum_{I=(i_1,\dots,i_n)} n_I [a_{i_1}|\dots |a_{i_{s-1}}] \otimes (a_{i_{s}} \wedge a_{i_{s+1}}) \otimes  [a_{i_{s+2}}|\dots |a_{i_{n}}]
\]
lies in $\mathrm{T}^{s-1}(A)\otimes \textup{CL}_2(X) \otimes \mathrm{T}^{n-s-1}(A).$ We define the space of cluster integrable symbols
$\textup{CL}_n(X)$  as the projection to $\textup{CoLie}_\bullet(A)$ of the  space of cluster adjacent and integrable elements in $\mathrm{T}^{n}(A)$.

For a Lie coalgebra $(\L, \Delta)$ we define iterated coproduct $\Delta^{[k]}\colon \L \lra \textup{CoLie}_k(\L)$
inductively. First, we have $\Delta^{[1]}=\Delta.$ Next, for $k>1$ we define  $\Delta^{[k]}$ as a composition of $\bigl( \Delta^{k-1} \otimes \textup{id} \bigr) \circ \Delta$ with the projection of  $\textup{CoLie}_{k-1}(\L) \otimes \L$ to $\textup{CoLie}_{k}(\L).$ In particular, we get a map
\[
\Delta^{[k-1]}\colon \L_k(\F)\lra  \textup{CoLie}_{k}(\F^{\times}_\Q) 
\]
which we call a symbol map and denote by $\mathcal{S}.$

 An element $ L \in \mathcal{L}(\F)$ for $\F=\mathbb{Q}(X)$ is called a  cluster polylogarithm on $X$ if its symbol $\mathcal{S}(L)$ lies in $ \textup{CL}(X).$ We expect that every cluster integrable symbol can be lifted {\it canonically} to a cluster polylogarithm.

\section{Cluster polylogarithms on the configuration space}\label{SectionClusterConfigurationSpace}

\subsection{Iterated integrals on the configuration space}  \label{SectionIIonConfSpace}

Consider a configuration space $\textup{Conf}_{m+2}$ of distinct points $t_0,\dots,t_{m+1}$ in $\mathbb{A}^1.$  It follows from the results of Arnold (\cite{Arn69})  that the algebraic de Rham cohomology algebra 
$
A_\bullet=H^\bullet_{dR}(\textup{Conf}_{m+2})
$ is generated  by forms
\[
\omega_{ij}=d\log(t_i-t_j)\in A_1
\] 
subject to relations
\[
\omega_{ij}\wedge \omega_{jk}+\omega_{jk}\wedge \omega_{ki}+\omega_{ki}\wedge \omega_{ij}=0.
\]
Algebra $A_\bullet$ is Koszul; its Koszul dual Lie coalgebra $\textup{L}_\bullet(\textup{Conf}_{m+2})$ can be described as follows.
It is a subcoalgebra of $\textup{CoLie}_\bullet(A_1)$ generated in degree $n$ by elements 
\[
\sum_{I=(i_1,\dots,i_n)} a_I [\omega_{i_1}|\dots|\omega_{i_n}]
\]
for $a_I\in \mathbb{Q}$ and $\omega_{i_k}\in A_1$
such that the following integrability conditions hold: 
\[
\sum_{I} a_I [\omega_{i_1}|\dots|\omega_{i_k}\wedge\omega_{i_{k+1}}|\dots|\omega_{i_n}]=0 \quad\text{ for }\quad 1\leq k \leq n-1.
\]

Iterated integral $\textup{I}^\L(t_{i_0};t_{i_1},\dots,t_{i_n};t_{i_{n+1}})$ is an element of $\mathcal{L}(\mathbb{Q}(t_0,\dots,t_{m+1})).$ Its symbol $\mathcal{S}\Bigl(\textup{I}(t_{i_0};t_{i_1},\dots,t_{i_n};t_{i_{n+1}})\Bigr)$ can be viewed  as an element of $\textup{L}_\bullet(\textup{Conf}_{m+2});$ for that we  identify $(t_i-t_j)\in \L_1$ with $\omega_{ij}.$ The following proposition is well-known and easy to prove.

\begin{proposition}
 The Lie coalgebra $\textup{L}_\bullet(\textup{Conf}_{m+2})$ is spanned by symbols of iterated integrals $\mathcal{S}\Bigl(\textup{I}^\L(t_{i_0};t_{i_1},\dots,t_{i_n};t_{i_{n+1}})\Bigr).$
\end{proposition}
\begin{proof}
The proof is based on the following observation, which is also the key idea of Arnold's computation of $H^\bullet_{dR}(\textup{Conf}_{m+2})$.  For each $i\in \{0,\dots,m+1\}$ we have a forgetful map $\partial_i\colon \textup{Conf}_{m+2}\lra \textup{Conf}_{m+1},$ which is a locally trivial fibre bundle with a fiber $\mathbb{A}^1\setminus\{t_0,\dots,\hat{t}_i\dots,t_{m+1}\}$.  From here, one can deduce a split exact sequence
\be \label{FormulaExactSequenceProjection}
0\lra \textup{L}_\bullet(\textup{Conf}_{m+1})\lra \textup{L}_\bullet(\textup{Conf}_{m+2})\stackrel{\textup{pr}_i}{\lra} \textup{CoLie}_{\bullet}(\langle f_0,\dots, \hat{f}_i,\dots, f_{m+1} \rangle)\lra 0,
\ee
where $ f_j=d \log(t_i-t_j)$ for $j\neq i$. The projection $
\textup{pr}_i\bigl([\omega_{i_1j_1}|\dots|\omega_{i_nj_n}]\bigr)$ equals to zero if $i\notin \{i_k,j_k\}$ for some $k$; otherwise we have  
$
\textup{pr}_i([\omega_{ij_1}|\dots|\omega_{ij_n}])=\left[f_{j_1}|\dots|f_{j_n}\right].
$
It follows that $\textup{L}_\bullet(\textup{Conf}_{m+2})$ is noncanonically isomorphic to a direct sum of cofree  Lie coalgebras on $2, 3,\dots, m+1$ cogenerators. 

We argue by induction on $m.$ Since
\[
\textup{pr}_{m+1}\Bigl(\mathcal{S}\bigl(\textup{I}^\L(t_{i_0};t_{i_1},\dots,t_{i_n};t_{m+1})\bigr)\Bigr)=[f_{i_0}|\dots|f_{i_n}],
\]
every element in $\textup{CoLie}_\bullet\left(\langle f_0,\dots, \hat{f}_i,\dots, f_{m+1} \rangle\right)$ is obtained as a projection of the symbol of an iterated integral. The statement follows from  the induction hypothesis and (\ref{FormulaExactSequenceProjection}).
\end{proof}

Consider the Lie coalgebra $\mathcal{L}^+_\bullet(\textup{Conf}_{m+2})\subseteq \mathcal{L}\bigl(\mathbb{Q}(t_0,\dots,t_{m+1})\bigr)$ spanned by iterated integrals 
\[ \label{Formula_increasing_II}
\textup{I}^\L\left(t_{i_0};t_{i_1},\dots,t_{i_n};t_{i_{n+1}}\right) 
\quad \text{for} \quad 0\leq i_0\leq i_1\leq \dots\leq i_{n+1}\leq m+1.
\]
It follows from (\ref{FormulaCorII}) that $\mathcal{L}^+_\bullet(\textup{Conf}_{m+2})$ is also  spanned by 
\be \label{Formula_increasing_Cor}
\textup{Cor}^\L\left(t_{i_0},t_{i_1},\dots,t_{i_n}\right) 
\quad \text{for} \quad 0\leq i_0\leq i_1\leq \dots\leq i_{n}\leq m+1.
\ee

Note that  the space $\mathcal{L}^+_\bullet(\textup{Conf}_{2})$ is spanned by elements 
$\textup{Cor}^\L(0,\dots,0,1,\dots,1).$

\begin{lemma}\label{LemmaEqualityOfSymbolsImpliesConstant} 
For $n>1$ the sequence 
\[
0\lra \mathcal{L}^+_n(\textup{Conf}_{2})\lra \mathcal{L}^+_n(\textup{Conf}_{m+2}) \stackrel{\mathcal{S}}{\lra} \textup{L}_n(\textup{Conf}_{m+2})
\]
is exact.
\end{lemma}
\begin{proof}
The following relation between correlators can be easily proved by induction:
\be \label{FormulaRelationsCor}
\sum_{0\leq i_0\leq i_1\leq \dots\leq i_{n}\leq m+1}\textup{Cor}^\L\left(t_{i_0},t_{i_1},\dots,t_{i_n}\right)=0 \quad \text{for} \quad m+2\leq n.
\ee
This identity implies that for each $i$ and $k$ element
\[
\textup{Cor}^\L\Bigl(t_{i_0},t_{i_1},\dots,t_{i_{n-k}},\underbrace{t_{i},\dots,t_{i}}_{k}\Bigr) 
\]
with $i_0\leq i_1\leq \dots\leq i_{n-k}<i$ can be expressed via similar elements with smaller $k$ and elements in $\mathcal{L}^+_n(\textup{Conf}_{2}).$ So, correlators 
\[
\textup{Cor}^\L\left(t_{i_0},t_{i_1},\dots,t_{i_{n-1}},t_{i}\right)\quad \text{for} \quad 0\leq i_0 \leq  \dots\leq i_{n-1}< i\leq m+1
\]
span the space $\mathcal{L}^+_n(\textup{Conf}_{m+2})/\mathcal{L}^+_n(\textup{Conf}_{2})$. On the other hand, their symbols are linearly independent, which follows immediately by applying projection maps $p_i$. 
\end{proof}

\subsection{Quadrangular polylogarithms}\label{SectionQuadrangular}
Quadrangular polylogarithms  were introduced in \cite[\S 5.1]{Rud20}; here, we repeat the definition. Consider a set  $\widetilde{\mathcal{C}}_{n,k}$ of all nondecreasing sequences $\bar{s}=(i_0,\dots,i_{n+k})$ of indices 
\[
0\leq i_0\leq i_1\leq \dots \leq i_{n+k}\leq 2n+1
\]
such that every even number $0\leq s\leq 2n+1$ appears in the sequence $\bar{s}$ at most once. Let $\mathcal{C}_{n,k}$ be the set of  sequences $\bar{s}\in \widetilde{\mathcal{C}}_{n,k},$ which contains at least one element in each pair $\{2i,2i+1\}$ for $0\leq i\leq n.$ For a sequence $\bar{s}\in \widetilde{\mathcal{C}}_{n,k}$ we define 
\[
\textup{sign}(\bar{s})=\begin{cases}
\ \ 1 & \text{\ if\ }  \bar{s} \text{\ contains an even number of even elements,}\\
-1 & \text{\ if\ } \bar{s} \text{\ contains an odd number of even elements}.
\end{cases}
\] 
For  $x_0, \dots, x_{2n+1}\in\F$, we define the quadrangular polylogarithm of weight $n+k$ by the formula 
\be\label{FormulaDefinitionQLi}
\textup{QLi}_{n+k}(x_0,\dots, x_{2n+1})=(-1)^{n+1}\sum_{\bar{s}\in\mathcal{C}_{n,k}} \textup{sign}(\bar{s})\textup{Cor}(x_{i_0},\dots,x_{i_{n+k}})\in \mathcal{L}_{n+k}.
\ee
Similarly, we define the {\it symmetrized} quadrangular polylogarithm of weight $n+k$ by the formula 
\be \label{FormulaDefinitionQLisym}
\textup{QLi}_{n+k}^{\textup{sym}}(x_0,\dots, x_{2n+1})=(-1)^{n+1}\sum_{\bar{s}\in\widetilde{\mathcal{C}}_{n,k}} \textup{sign}(\bar{s})\textup{Cor}(x_{i_0},\dots,x_{i_{n+k}})\in \mathcal{L}_{n+k}.
\ee

Clearly, $\textup{QLi}_{n+k}$ and  $\textup{QLi}_{n+k}^{\textup{sym}}$ lie in the space $\mathcal{L}^+_{n+k}(\textup{Conf})$ defined in \S \ref{SectionIIonConfSpace}. These functions are closely related.
\begin{lemma} \label{LemmaSymmetrizedQuadrangularPolylogarithm} The following formula holds:
\begin{align*}
   \textup{QLi}_{n+k}^{\textup{sym}}&(x_0,\dots, x_{2n+1})\\
 &=\sum_{0\leq i_0<\dots<i_{r}\leq n} (-1)^{n-r} \textup{QLi}_{n+k}(x_{2i_0},x_{2i_0+1},x_{2i_1},x_{2i_1+1},\dots, x_{2i_r},x_{2i_r+1}).    
\end{align*}
\end{lemma}
\begin{proof}
This follows directly from the definitions of $\textup{QLi}_{n,k}$ and $\textup{QLi}_{n,k}^{\textup{sym}}.$
\end{proof}

It is convenient to consider a generating function
\[
\textup{QLi}(x_0,\dots, x_{2n+1})=\sum_{k\geq 0} \textup{QLi}_{n+k}(x_0,\dots, x_{2n+1})\in \L;
\]
similarly for $\textup{QLi}^{\textup{sym}}.$
We introduce the following notation:
\[
\textup{QLi}_{n+k}^{(-)^s}(x_0,x_1,\dots, x_{2n},x_{2n+1})=
\begin{cases}
\textup{QLi}_{n+k}(x_0,x_1,\dots, x_{2n},x_{2n+1})	& \text{\ if \ } s  \text{\ is even, \ }\\ 
-\textup{QLi}_{n+k}(x_1,x_2,\dots, x_{2n+1},x_0)	& \text{\ if \ } s  \text{\ is odd. \ }
\end{cases}
\]
We have the following formula for the coproduct  of quadrangular polylogarithms (see \cite[Theorem 5.5]{Rud20}):
\be \label{ClusterPolylogarithmCoproduct}
\begin{split}
	\Delta^{\mathcal{L}}&\textup{QLi}(x_0,\dots, x_{2n+1})=\\
	&\sum_{\substack{0\leq i<j\leq2n+1\\j-i=2s+1}} \textup{QLi}(x_0,\dots,x_{i},x_{j},\dots,x_{2n+1})\wedge \textup{QLi}^{(-)^{i}}(x_{i},\dots,x_{j}).\\
\end{split}
\ee

From (\ref{ClusterPolylogarithmCoproduct}) it is easy to deduce that $\textup{QLi}_{n+k}$ and  $\textup{QLi}_{n+k}^{\textup{sym}}$ are projectively invariant, see \cite[Proposition 5.7]{Rud20}. Thus, it does not change if we substitute every term $x_i-x_j$ via a Pl{\"u}cker coordinate $\Delta_{ij}.$ The key observation is that quadrangular polylogarithms are cluster polylogarithms on $\textup{Gr}(2,2n+2).$ 

\begin{proposition} \label{PropositionSymbolQuadrangular}
The symbol $\mathcal{S}(\textup{QLi}_{n+k}(x_0,\dots, x_{2n+1}))$ lies in $\textup{CL}_{n+k}\bigl(\mathfrak{M}_{0,2n+2}\bigr).$
\end{proposition}	
\begin{proof}
The cluster variables $\Delta_{ij}$ correspond to diagonals of the $(n+2)$-gon with vertices labeled by points $x_0,\dots, x_{2n+1}$. We call $\Delta_{ij}$ and $\Delta_{i'j'}$  weakly separated if the corresponding diagonals have no common interior points. A collection of pairwise weekly separated cluster variables is contained in some cluster.  The statement follows  immediately from (\ref{ClusterPolylogarithmCoproduct}), because all Pl{\"u}cker coordinates appearing in $\textup{QLi}(x_0,\dots,x_{i},x_{j},\dots,x_{2n+1})$ are weakly separated from all  Pl{\"u}cker coordinates appearing in $\textup{QLi}^{(-)^{i}}(x_{i},\dots,x_{j}).$
\end{proof}

The main property of quadrangular polylogarithms is the {\it quadrangulation formula} proved in  \cite[Theorem 1.2]{Rud20}. Consider a convex $(2n+2)$-gon $\P$ with vertices labeled by points $x_0,\dots, x_{2n+1}\in \mathbb{P}^1_\F.$ Every quadrangle inside $\P$ with vertices $x_{i_0}, x_{i_1},x_{i_2},x_{i_3}$  determines a cross-ratio $[x_{i_0}, x_{i_1},x_{i_2},x_{i_3}]\in \F^{\times}.$  Let $\mathcal{Q}(\P)$ be the set of quadrangulations of $\P;$  for a quadrangulation $Q\in \mathcal{Q}$ denote by $\mathrm{t}_Q$ the dual tree. We denote by $\Li_k(\mathrm{t})$ a  certain sum of multiple polylogarithms of weight $n+k$ and depth at most $n$ evaluated at products of cross-ratios, corresponding to the vertices of $\mathrm{t}$ (for details, see  \cite[\S\S 4.1--4.3]{Rud20}). Then the following formula holds:
\be \label{FormulaQuadrangulationFormula}
\textup{QLi}_{n+k}(x_0,\dots,x_{2n+1})=\sum_{Q\in \mathcal{Q}(\P)} \Li_{k}(\mathrm{t}_Q).
\ee
It follows that  $\textup{QLi}_{n+k}(x_0,\dots, x_{2n+1})$ has depth at most $n$ for any $k\geq 0.$

\begin{proposition}\label{SymmertyQuadrangularPolylogarithm}
The following equations hold:
\be \label{FormulaSymmetry}
\textup{QLi}_{n+k}^{\textup{sym}}(x_1,x_2,\dots, x_{2n+1},x_0)=(-1)^{n+k}\textup{QLi}^{\textup{sym}}_{n+k}(x_0,x_1,x_2,\dots, x_{2n+1}).
\ee
\end{proposition}
\begin{proof}
First, we show that (\ref{FormulaSymmetry}) holds on the level of symbols. It is sufficient to check  that 
\be \label{FormulaProjectionSymmetry}
\textup{pr}_0\left(\textup{QLi}_{n+k}^{\textup{sym}}(x_1,x_2,\dots, x_{2n+1},x_0)-(-1)^{n+k}\textup{QLi}^{\textup{sym}}_{n+k}(x_0,x_1,x_2,\dots, x_{2n+1})\right)=0;
\ee
the rest follows from cyclic symmetry. Identity (\ref{FormulaProjectionSymmetry}) follows easily from (\ref{ClusterPolylogarithmCoproduct}).

By Lemma \ref{LemmaEqualityOfSymbolsImpliesConstant}, the equality of symbols implies that 
\[
\textup{QLi}_{n+k}^{\textup{sym}}(x_1,x_2,\dots, x_{2n+1},x_0)-(-1)^{n+k}\textup{QLi}^{\textup{sym}}_{n+k}(x_0,x_1,x_2,\dots, x_{2n+1})\in \mathcal{L}^+_n(\textup{Conf}_{2}),
\]
so is constant. Specializing to the point $x_0=\dots=x_{2n+1},$ we obtain the statement.
\end{proof}

Now we are ready to formulate the main functional equation satisfied by quadrangular polylogarithms.
\begin{proposition}\label{PropositionMainFunctionalEquation} For $n<N-1$ the following equality holds:
\begin{equation}\label{FormulaMainFunctionalEquation}
\sum_{0\leq i_0<\dots <i_{2r+1}\leq N}(-1)^{i_0+\dots+i_{2r+1}}\textup{QLi}_{n}^{\textup{sym}}(x_{i_0},x_{i_1},\dots, x_{i_{2r+1}})=0.
\end{equation} 
\end{proposition}

\begin{proof} Remarkably,  the statement follows directly from (\ref{FormulaDefinitionQLisym}): no properties of correlators are used in the proof. 
Consider a map 
\[
\textup{Cor}\colon \mathbb{Z}[[t_0,\dots,t_N]] \lra \mathcal{L}(\F)
\]
sending formal power series $\sum n_{I}t_1^{k_0}\dots t_N^{k_N}$ to
\[
\textup{Cor}\left(\sum n_{k_0,\dots,k_N}t_0^{k_0}\dots t_N^{k_N}\right) = \sum n_{k_0,\dots,k_N} \textup{Cor}\Bigl(\underbrace{x_0,\dots,x_0}_{k_0},\dots,\underbrace{x_N,\dots,x_N}_{k_N}\Bigr).
\]

By (\ref{FormulaDefinitionQLisym}), we have
\[
\textup{QLi}^{\textup{sym}}(x_{i_0},x_{i_1},\dots, x_{i_{2r+1}})=\textup{Cor}\left(\frac{(1-t_{i_0}) (1-t_{i_2}) \dots(1-t_{i_{2r}} )\ \ \  }{(1-t_{i_1}) (1-t_{i_3}) \dots (1-t_{i_{2r+1}})}\right).
\]
We need to show that the power series 
\begin{align*}
\Psi(x_0,\dots,x_N)=\sum_{0\leq i_0<\dots <i_{2r+1}\leq N}(-1)^{i_0+\dots+i_{2r+1}}\frac{(1-t_{i_0}) (1-t_{i_2}) \dots(1-t_{i_{2r}} )\ \ \  }{(1-t_{i_1}) (1-t_{i_3}) \dots (1-t_{i_{2r+1}})}
\end{align*}
has no terms of degree less than $N.$ This follows from an elementary identity
\[
\Psi(t_0,\dots,t_N)=
\frac{(t_{0}-t_1)(t_{1}-t_{2})\dots (t_{N-1}-t_{N})}{(1-t_{1})\ (1-t_2)\ \dots \ (1-t_{N})}.
\]
\end{proof}

\subsection{The space of quadrangular polylogarithms}

For $n\geq 2$,  consider a $\mathbb{Q}$-vector space $\mathcal{Q}_n(m)\subseteq \mathcal{L}^+_n(\textup{Conf}_{m+1})$ spanned by quadrangular polylogarithms 
\[
\textup{QLi}_{n}(x_{i_0},\dots,x_{i_{2r+1}})
\]
for $0\leq i_0< \dots< i_{2r+1}\leq m.$ Our first goal is to show that the sequence of spaces  $\mathcal{Q}_n(0), \mathcal{Q}_n(1), \dots$ has the structure of a cocyclic vector space in the sense of Connes (see \cite{Con83}, \cite{CC15}).

\begin{lemma} For a periodic nondecreasing map $\alpha\colon [m_1] \lra [m_2]$ an element
\[
 \alpha\textup{QLi}_{n}(x_{i_0},\dots,x_{i_{2r+1}}):=\textup{QLi}_{n}\bigl(x_{ \alpha(i_0)},\dots,x_{ \alpha(i_{2r+1})}\bigr)
\]
lies in $\mathcal{Q}_n(m_2).$
\end{lemma}
\begin{proof}
Any periodic nondecreasing map is a composition of  coface maps $\delta_i\colon [m-1]\lra [m]$  for $0\leq i \leq m,$ codegeneracy maps $\sigma_i\colon [m+1]\lra [m]$  for $0\leq i \leq m$, and cyclic shifts  $\tau\colon [m]\lra [m] $ defined by formulas
\[
\delta_i(j)=
\begin{cases} 
j & \text{ if } j<i,\\
j+1 & \text{ if } j\geq i
\end{cases}, \ \ \ 
\sigma_i(j)=
\begin{cases} 
j & \text{ if } j\leq i,\\
j-1 & \text{ if } j> i
\end{cases}, \ \text{ and } \ 
\tau(j)=(j+1) \text{ mod } m+1.
\]
Thus it is sufficient to show that $\textup{QLi}_{n}(x_{ \alpha(i_0)},\dots,x_{ \alpha(i_{2r+1})})$ lies in $\mathcal{Q}_n(m_2)$ for $\alpha=\delta_i, \sigma_i$ or $\tau.$ For $\delta_i$, this is obvious.  For $\tau_m$, the statement follows from Proposition \ref{SymmertyQuadrangularPolylogarithm}. For $\sigma_{2i}$, we have
\[
\sigma_{2i}\textup{QLi}_{n+k}(x_{0},\dots,x_{2n+1})=0
\]
by  (\ref{FormulaDefinitionQLi}). Finally, we have $\sigma_{2i+1} =\tau ^{-1}\sigma_{2i}\tau,$ so for $\sigma_{2i+1}$ the statement holds as well.
\end{proof}

A cocyclic vector space is a cosimplicial abelian group, so we can apply the (cosimplicial) Dold-Kan correspondence. Recall that for a cosimplicial abelian group $A(0), A(1),\dots$ one defines normalized cochain complex  $CN(A),$ where $CN(A)^m$ equals the quotient of $A(m)$ by the images of the coface maps $\delta_0,\dots,\delta_{m-1}.$ The map $(-1)^m\delta_{m}$ induces the differential $CN(A)^{m-1}\lra CN(A)^m.$ The Dold-Kan correspondence implies that the group  $CN(A)^m$ is isomorphic to the group 
\[
\{a\in A(m)\mid \sigma_ja=0 \text{ for } 0\leq j\leq m-1\}.
\]
Moreover, we have a canonical isomorphism
\be \label{FormulaDoldKanSum}
\bigoplus_{[m]\twoheadrightarrow [k]}CN\left(A\right)^k\cong A_m.
\ee

\begin{proposition}\label{PropositionCNQLi} The space $CN\left(\mathcal{Q}_n\right)^{m}$ has dimension one for $3\leq m\leq n+1$ and zero otherwise. 
\end{proposition}
\begin{proof} Consider the space $CN\left(\mathcal{Q}_n\right)^{m}$ which is the quotient of $\mathcal{Q}_n(m)$ by the images of the coface maps $\delta_0,\dots,\delta_{m-1}.$ Proposition \ref{PropositionMainFunctionalEquation} 
implies that it is zero except for $3\leq m\leq n+1,$ where it is generated by
$\textup{QLi}_{n}(x_0,\dots,x_m)$ for an odd $m$ and by $\textup{QLi}_{n}(x_0,\dots,x_{m-1})$
for an even $m.$  Looking at the symbols of these functions, it is easy to show that these elements are nonzero in $CN\left(\mathcal{Q}_n\right)^{m}$.
\end{proof}

\begin{corollary} The following formula holds:
\be \label{FormulaDimensionQLi}
 \textup{dim}(\mathcal{Q}_n(m))={m \choose 3}+{m \choose 4}+\dots+{m \choose n+1}.
\ee
\end{corollary}
\begin{proof}
The statement follows from Proposition \ref{PropositionCNQLi} and (\ref{FormulaDoldKanSum}).
\end{proof}

\subsection{Proof of Theorem  \ref{TheoremClusterTypeA}} \label{SectionSpaceClusterIntegrable}
For $n\geq 2$, the sequence of vector spaces 
\[
\mathcal{C}_n(m)=\textup{CL}_n \left(\mathfrak{M}_{0,m+1}\right), \ m\geq 0
\]
 has a structure of the cocyclic object in the category of $\mathbb{Q}$-vector spaces. For a periodic nondecreasing map
$\alpha\colon [m_1] \lra [m_2]$,
we define the corresponding map on Pl{\"u}cker coordinates by the  formula 
\[
\alpha(\Delta_{ij})=
\begin{cases}
\Delta_{\alpha(i)\alpha(j)} & \text{ if } \alpha(i)\neq \alpha(j),\\
0 & \text{ if } \alpha(i)=\alpha(j)
\end{cases}
\]
and extend this map to the space  $\textup{CL}_{n}$ in a natural way. It is easy to see that a periodic nondecreasing map sends cluster adjacent symbols  to cluster adjacent. Furthermore, $\alpha\Bigl(\left[x_{i_1},x_{i_2},x_{i_3},x_{i_4}\right]\wedge (1-\left[x_{i_1},x_{i_2},x_{i_3},x_{i_4}\right])\Bigr)$ vanishes if  at least two indices $\alpha(i_1), \alpha(i_2), \alpha(i_3), \alpha(i_4)$ coincide. Otherwise, it equals to 
\[
\left[x_{\alpha(i_1)},x_{\alpha(i_2)},x_{\alpha(i_3)},x_{\alpha(i_4)}\right]\wedge\bigl(1-\left[x_{\alpha(i_1)},x_{\alpha(i_2)},x_{\alpha(i_3)},x_{\alpha(i_4)}\right]\bigr).
\]
Thus 
$
f\left(\textup{CL}_{2} \left(\mathfrak{M}_{0,m_1+1}\right) \right) \subseteq \textup{CL}_{2} \left(\mathfrak{M}_{0,m_2+1}\right),
$
so integrable symbols are mapped to integrable symbols and
$\mathcal{C}_n(m)$ is endowed with a structure of a cocyclic vector space.

By Proposition \ref{PropositionSymbolQuadrangular}, the symbol is a map of cocyclic vector spaces
\[
\mathcal{S}\colon \mathcal{Q}_n(m) \lra \mathcal{C}_n(m).
\] To prove Theorem \ref{TheoremClusterTypeA}, we need to show that $\mathcal{S}$ is an isomorphism. Lemma \ref{LemmaEqualityOfSymbolsImpliesConstant} implies that $\mathcal{S}$ is injective, so it is sufficient to check that 
\be \label{FormulaEqualityDimensions}
\textup{dim}\left(\mathcal{Q}_n(m)\right)\geq \textup{dim}\left(\mathcal{C}_n(m)\right) \quad \text{for}\quad m\geq 0.
\ee

Consider the projection map 
\[
\textup{pr}_m\colon\mathcal{C}_n(m)\lra \textup{CoLie}_n\bigl(\langle f_0,\dots,f_{m-1} \rangle\bigr)
\]
defined in \S \ref{SectionIIonConfSpace}. Clearly, $\textup{pr}_m$ vanishes on the subspace $\delta_m \mathcal{C}_n(m-1) \subseteq \mathcal{C}_n(m).$ Elements in the image of $\textup{pr}_m$ satisfy the following two properties. First, they are invariant under transformations $T_f$ sending each $f_i$ to $f_i+f$ for  $f\in \textup{CoLie}_1.$
Second,  they can be expressed as linear combinations of  $f_{i_1}\otimes \dots \otimes f_{i_n}$ with $i_1\leq i_2\leq \dots \leq i_n$. Denote the subspace  $\textup{CoLie}_n\bigl(\langle f_0,\dots,f_{m}\rangle\bigr)$ consisting of elements satisfying these two properties  by $\textup{Inv}_n(m).$
We get the following exact sequence: 
\be \label{ExactSequenceInv}
0\lra \partial_n \mathcal{C}_n(m-1) \lra \mathcal{C}_n(m)\stackrel{p}{\lra} \textup{Inv}_n(m).
\ee

Here is an example. The space $\textup{Inv}_2(3)$ is spanned by one element
\[
a_2(2)=[f_0|f_1]-[f_0|f_2]+[f_1|f_2].
\]
To see that $a_2(2)\in\textup{Inv}_2(2)$, it is sufficient to check that $a_2(3)$ is invariant under $T_f:$
\begin{align*}
    T_f&([f_0|f_1]-[f_0|f_2]+[f_1|f_2])\\
    &=[f_0+f|f_1+f]-[f_0+f|f_2+f]+[f_1+f|f_2+f]\\
    &=[f_0|f_1]-[f_0|f_2]+[f_1|f_2]+[f|f_1]+[f_1|f]+[f|f]\\  
    &=[f_0|f_1]-[f_0|f_2]+[f_1|f_2].\\
\end{align*}

\begin{lemma}\label{LemmaConormalizedInvariantPolynomials}
The dimension of the space $\textup{Inv}_n(m)$ is equal to 
\[
{m \choose 2}+{m \choose 3}+\dots+{m \choose n}.
\]
\end{lemma}

\begin{proof}
It is easy to see that the sequence of vector spaces   $\textup{Inv}_n(m)$ for $m\geq 0$
is a cosimplicial vector space: a  nondecreasing map  $\alpha\colon [m_1] \lra [m_2]$ induces the corresponding map  
\be\label{FormulaDoldKan}
\alpha[f_{i_1}| \dots | f_{i_n}]=[f_{\alpha(i_1)}| \dots | f_{\alpha(i_n)}].
\ee
So, it is sufficient to prove that  $CN\left(\textup{Inv}_n\right)^m$
has dimension one for $2\leq m\leq n$ and zero otherwise. After that, the statement would follow from (\ref{FormulaDoldKanSum}).

We prove a more precise statement:  the space $CN\left(\textup{Inv}_n\right)^m$ is generated by an element
\[
\sum_{n_0+\dots+n_m=n-m}f_0^{\otimes n_0}\otimes (f_0-f_1)\otimes f_1^{\otimes n_1}\otimes (f_1-f_2)\otimes f_2^{\otimes n_2}\otimes \dots \otimes (f_{m-1}-f_m)\otimes  f_m^{\otimes n_m},
\]
which we denote $a_n(m).$
First, we show that $a_n(m)$ lies in the space $CN\left(\textup{Inv}_n\right)^m.$ The only nontrivial part is to show that $a_n(m)$ is invariant under $T_f.$ That follows from the following identity in the tensor algebra, which can be easily checked:
\[
T_f (a_n(m))=a_n(m)+f\shuffle a_{n-1}(m)+\dots+f^{\otimes  (n-m)}\shuffle a_{m}(m)
\]
It follows that in the Lie coalgebra, we have $T_f (a_n(m))=a_n(m)$.

Next, consider a subspace $W_n$ of  $\textup{CoLie}_n\bigl(\langle f_0,\dots,f_{m}\rangle\bigr)$ spanned by elements  $[f_{i_1}| \dots | f_{i_n}]$ with $i_1\leq \dots \leq i_n$.
By the theory of Lyndon bases, $f_i^{\otimes n}=0$ and elements
\[
f_{0}^{\otimes n_0}\otimes f_{1}^{\otimes n_1} \otimes \dots \otimes f_{m}^{\otimes n_m}
\]
with at least two nonzero $n_i$'s are linearly independent. Thus $W_n$ is isomorphic to the degree $n$ graded component of the ring
\[
\dfrac{\mathbb{Z}[t_0,t_1,\dots,t_m]}{(t_0^n,\dots,t_m^n)}.
\]
The degeneracy maps $\sigma_i\colon W_n\lra W_{n-1}$
are defined by formulas 
\[
\sigma_i\bigl(P(t_0,\dots,t_m)\bigr)=P(t_0,\dots,t_{i},t_{i},\dots,t_{m-1}).
\]
Thus $\sigma_i(P)=0$ if and only if $(t_{i}-t_{i+1})|P.$ It follows that an element of $W$ lies in the kernel of all degeneracy maps if and only if it is divisible by $(t_0-t_1)(t_1-t_2)\dots (t_{m-1}-t_{m}).$ Thus every element in $CN\left(\textup{Inv}_n\right)^m$ has a form
\[
\sum_{n_0+\dots+n_m=n-m} a_{n_0,\dots,n_m}f_0^{\otimes n_0}\otimes (f_0-f_1)\otimes f_1^{\otimes n_1}\otimes (f_1-f_2)\otimes \dots \otimes (f_{m-1}-f_m)\otimes  f_m^{\otimes n_m}.
\]
 We have
\[
T_f\Bigl(f_i^{\otimes n_i}\Bigr)=f_i^{\otimes n_i}+ f \shuffle f_i^{\otimes (n_i-1)}+ \Bigl( \text{ terms with at least two $f$'s.}\Bigr)
\]
Consider elements
$
X_i=x_1\dots x_{i}x x_{i+1}\dots x_n 
$
for $0\leq i\leq n$, which lie in a cofree Lie coalgebra. It is easy to see that if $x\notin \langle x_1,\dots,x_m\rangle$ then $X_i$ satisfy a unique equation up to rescaling, namely, $\sum_{i=0}^{n} X_i=0.$  So, if $a$ is invariant under translations then 
$a_{n_0,\dots, n_i-1,\dots, n_j+1,\dots,n_m}=a_{n_0,\dots,n_m}$
for any $i < j.$ This implies that all coefficients $a_{n_0,\dots,n_m}$ are the same, from where the statement follows.
\end{proof}

We are ready to finish the proof of Theorem \ref{TheoremClusterTypeA}. From Lemma \ref{LemmaConormalizedInvariantPolynomials} and (\ref{ExactSequenceInv}) it follows that 
\[
\textup{dim}\left( \mathcal{C}_n(m)\right)-\textup{dim}\left( \mathcal{C}_n(m-1)\right)\leq {m-1 \choose 2}+{m-1 \choose 3}+\dots+{m-1 \choose n},
\]
so 
\[
\textup{dim}\left( \mathcal{C}_n(m)\right)\leq {m \choose 3}+{m \choose 4}+\dots+{m \choose n+1}=\textup{dim}\left( \mathcal{Q}_n(m)\right).
\]
So, we have proved (\ref{FormulaEqualityDimensions}), which implies Theorem \ref{TheoremClusterTypeA}. We also obtain the following corollary.

\begin{corollary}\label{CorollaryProjectiveInvariantII}
The subspace of the space $\textup{L}^+_\bullet(\textup{Conf}_{m+2})$  consisting of elements which are invariant under projective transformations is spanned by quadrangular polylogarithms. 
\end{corollary}

\section{Higher Gangl formula in weight six}\label{SectionHigherGangl}

\subsection{The depth conjecture} \label{SectionDepthConjecture}
Let $\F$ be a field and $(\L(\F),\Delta)$ --- Lie coalgebra of multiple polylogarithms. Assume that $\Delta=\sum_{1\leq i\leq j}\Delta_{ij}$ for  $\Delta_{ij}\colon \mathcal{L}_{i+j}(\mathrm{F}) \longrightarrow \mathcal{L}_i(\mathrm{F}) \wedge \mathcal{L}_j(\mathrm{F})$. The truncated coproduct is a map $\overline{\Delta}\colon \mathcal{L}(\mathrm{F})\longrightarrow \bigwedge^2\mathcal{L}(\mathrm{F})$ defined by the formula 
\[
\overline{\Delta}=\sum_{2\leq i \leq j} \Delta_{ij}.
\]
In other words, $\overline{\Delta}$ is obtained from $\Delta$ by omitting the component $\mathcal{L}_{1}(\mathrm{F}) \wedge \mathcal{L}_{n-1}(\mathrm{F})$ from the coproduct.

\begin{proposition}
For $k\geq 2$ iterated truncated coproduct 	$\overline{\Delta}^{[k-1]}$ vanishes on  $\mathcal{D}_{k-1}\L_\bullet(\F)$ and defines a map 
\[
\overline{\Delta}^{[k-1]}\colon \mathrm{gr}^\mathcal{D}_k\L(\F) \lra \textup{CoLie}_k\left(\bigoplus_{n\geq 2}\mathcal{B}_n(\F)\right).
\]
\end{proposition}
\begin{proof}
Recall the formula for the coproduct of an iterated integral:
\[
\begin{split}
	&\Delta^{\mathcal{L}} \textup{I}^{\mathcal{L}}(x_0;x_1,\dots,x_n;x_{n+1})
	=\sum_{i<j} \textup{I}^{\mathcal{L}}(x_{i};x_{i+1},\dots,x_{j-1};x_{j})\wedge  \textup{I}^{\mathcal{L}}(x_{0};x_{i},\dots, x_{j} ; x_{n+1}).
\end{split}
\]
 We can assume that $x_0=0$ and $x_{n+1}\neq0$. If the iterated integral 
$\textup{I}^{\mathcal{L}}(x_0;x_1,\dots,x_n;x_{n+1})$
has depth $k,$ then the number of nonzero terms in the  sequence $x_0,\dots,x_{n+1}$ is at most $k+1.$ It is sufficient to show that if the  iterated integral  $\textup{I}^{\mathcal{L}}(x_{0};x_{i},\dots, x_{j} ; x_{n+1})$ has depth at least $k-1$ then the  iterated integral  $\textup{I}^{\L}(x_{i};x_{i+1},\dots,x_{j-1};x_{j})$ has depth at most one. 

The number of nonzero terms in the sequence $x_0,x_i,\dots,x_j,x_{n+1}$ is at least k. If $x_i=x_j=0,$ there is nothing to prove. If  $x_i\neq 0, x_j=0$ or $x_j\neq 0, x_i=0,$  the number of nonzero terms in the sequence $x_i,\dots,x_j$ is at most two. Shuffle relations imply that the iterated integral
\[
\textup{I}^{\mathcal{L}}(0; 0,\dots,0,a,0,\dots,0;b)
\]
has depth at most one, from where the statement follows. If $x_i\neq 0, x_j\neq 0,$ the number of nonzero terms in the sequence $x_i,\dots,x_j$  is at most three and the statement follows from the fact that 
\[
\textup{I}^{\mathcal{L}}(a; 0,\dots,0,b,0,\dots,0;c)=\textup{I}^{\mathcal{L}}(a; 0,\dots,0,b,0,\dots,0;0)+\textup{I}^{\mathcal{L}}(0; 0,\dots,0,b,0,\dots,0;c)
\]
has depth at most one. 
\end{proof}

Now we are ready to formulate a precise version of Conjecture \ref{ConjectureDepth}.

\begin{conjecture}[Depth conjecture] For $k\geq 2$ the map
\[
\overline{\Delta}^{[k-1]}\colon \mathrm{gr}^\mathcal{D}_k\L(\F) \lra \textup{CoLie}_k\left(\bigoplus_{n\geq 2}\mathcal{B}_n(\F)\right)
\]
is an isomorphism.
\end{conjecture}

The surjectivity of the map $\overline{\Delta}^{[k-1]}$ is proved in \cite{CGRR22}. It is easy to see that $\overline{\Delta}^{[k-1]}$ vanishes on $\L_n$  for $2k>n.$  In \cite[Theorem 1]{Rud20} it is proved that $\mathrm{gr}^\mathcal{D}_k\L_{n}(\F)=0$ for $2k>n.$

In what follows, we assume that $n=2k.$ The weight $2k$ component of the Lie coalgebra
$
\textup{CoLie}_k\left(\bigoplus_{n\geq 2}\mathcal{B}_n(\F)\right)
$
equals to
$\textup{CoLie}_k\left(\mathcal{B}_2(\F)\right).$ In this case, the depth conjecture states that the map 
\be \label{FormulaIteratedCoproductCritical}
\overline{\Delta}^{[k-1]}\colon \mathrm{gr}^\mathcal{D}_k\L_{2k}(\F)\lra \textup{CoLie}_k\left(\mathcal{B}_2(\F)\right)
\ee
is an isomorphism. 

Map (\ref{FormulaIteratedCoproductCritical}) is surjective. Indeed,
\[
	\overline{\Delta}^{[k-1]} \Li_{k;1,\dots,1}^{\L}(a_1,\dots,a_k)=\Li_2^{\L}(a_1)\otimes\dots\otimes  \Li_2^{\L}(a_k),
\]
and such elements span $\textup{CoLie}_k\left(\mathcal{B}_2(\F)\right).$ To prove injectivity, we need to construct a map in the other direction. Proposition \ref{PropositionWeightTwo} implies that there exists a presentation of the Bloch group $\mathcal{B}_2(\F)$ by generators and relations:
\[
0\lra \mathrm{R}_2(\F)\lra \Q[\F^{\times}]\lra \mathcal{B}_2(\F)\lra 0,
\]
where $\{a\}\in \Q[\F^{\times}]$ is mapped to $\Li_2^\L(a)$ and $\mathrm{R}_2(\F)$ is generated by elements
\[
\sum_{i=0}^4 (-1)^i\bigl\{[x_0,\dots,\widehat{x_i},\dots,x_4]\bigr\}\in \Q[\F^{\times}].
\]
It is easy to see that we have an exact sequence
\[
\mathrm{R}_2(\F) \otimes \textup{CoLie}_{k-2}\left(\Q[\F^{\times}]\right)\lra \textup{CoLie}_{k}\left(\Q[\F^{\times}]\right)\lra \textup{CoLie}_k\left(\mathcal{B}_2(\F)\right)\lra 0.
\]

We define a map 
\be \label{FormulaInverseMap}
\mathcal{I}\colon \textup{CoLie}_k\left(\Q[\F^{\times}]\right)\lra  \mathrm{gr}^\mathcal{D}_k\L_{2k}(\F)
\ee
by the  formula
$
\mathcal{I}(\{a_1\}\otimes \dots \otimes\{a_k\} )=\Li_{k;1,\dots,1}^{\L}(a_1,\dots,a_k).
$
This map is well-defined because of the quasi-shuffle relations for multiple polylogarithms (\cite[Proposition 3.10]{Rud20}). To prove injectivity of  (\ref{FormulaIteratedCoproductCritical}) it remains to show that $\mathcal{I}$ vanishes on the space 
$\mathrm{R}_2(\F) \otimes \textup{CoLie}_{k-2}\left(\Q[\F^{\times}]\right).$ Thus, the depth conjecture for $n=2k$ is equivalent to Conjecture \ref{ConjectureDepthObstr}.

We break Conjecture \ref{ConjectureDepthObstr} into two parts, generalizing the formulas of Zagier and the formula of Gangl in weight four.
\begin{conjecture}[Higher Zagier formulas]\label{ConjectureHigherZagier}
Elements 
\begin{align}
&\Li_{k;1,1,\dots,1}^{\L}(a_1,a_2\dots,a_k)+\Li_{k;1,1,\dots,1}^{\L}(1-a_1,a_2\dots,a_k)\label{FormulaZagier1}\in \L_{2k}\\
&\Li_{k;1,1,\dots,1}^{\L}(a_1,a_2\dots,a_k)+\Li_{k;1,1,\dots,1}^{\L}\left(\frac{1}{a_1},a_2\dots,a_k\right)\in \L_{2k}\label{FormulaZagier2}
\end{align}
have depth at most $k-1.$
\end{conjecture}
Conjecture \ref{ConjectureHigherZagier} was proved for $k=2$ by Zagier (see \S \ref{SectionHigherGanglFour} for another approach) and is open for $k\geq 3.$ Next, denote by $\mathcal{G}_k$ the quotient of $\mathrm{gr}^\mathcal{D}_k\L_{2k}(\F)$ by the subspace spanned by elements 
(\ref{FormulaZagier1}) and (\ref{FormulaZagier2}). Clearly, $\overline{\Delta}^{[k-1]}$ vanishes on this subspace.
\begin{conjecture}[Higher Gangl formula]\label{ConjectureHigherGangl} 
The map 
\[
\overline{\Delta}^{[k-1]}\colon \mathcal{G}_k \lra \textup{CoLie}_k\left(\mathcal{B}_2(\F)\right)
\]
is an isomorphism. 
\end{conjecture}
For $k=2,$ Conjecture \ref{ConjectureHigherGangl} was proved by Gangl in \cite{Gan16}, see \S \ref{SectionHigherGanglFour} for another approach. In  \S \ref{SectionHigherGanglSix}, we prove Conjecture \ref{ConjectureHigherGangl} for $k=3$.

\subsection{Functional equation in $\mathcal{G}_k$} \label{SectionFunctionalHigherGangl}
In this section, we prove a functional equation in the space $\mathcal{G}_k$, which was defined in \S \ref{SectionDepthConjecture}. 

Let $\textup{P}_\F$ be the quotient of $\Q[\mathbb{P}^1_\F]$ by the subspace spanned by
$\{0\},\{1\},\{\infty\},$ and
\[
 \{x\}+\left\{x^{-1}\right\},\  \{x\}+\{1-x\} \quad \text{for} \quad x\in \F^{\times}.
\]

It is easy to see that we have a well-defined map 
\[
\mathcal{I}\colon \textup{CoLie}_k(\textup{P}_\F) \lra \mathcal{G}_k
\]
sending $\{a_1\} \otimes \dots \otimes \{a_k\}$ to $\Li_{k;1,\dots,1}^{\L}(a_1,\dots,a_k).$
For $x,y \in \textup{CoLie}_k(\textup{P}_\F)$, we write $x \equiv y$ if $x-y\in \textup{Ker}(\mathcal{I}).$

Next, we define  inductively certain elements 
\[
Q(x_0,\dots,x_{2k+1}), \ \  S(x_0,\dots,x_{2k+2})\in\textup{CoLie}_k(\textup{P}_\F).
\]
For $k=1$ we put
\begin{align*}
Q(x_0,\dots,x_3)&= \bigl\{[x_0,x_{1},x_{2},x_{3}]\bigr\}; \\
S(x_0,\dots,x_4)&=\sum_{i=0}^4  (-1)^i \bigl\{[x_0,\dots ,\hat{x}_i,\dots,x_4]\bigr\}.
\end{align*}
For $k\geq 2$ we have
\begin{align*}
    Q(x_0,&\dots,x_{2k+1})\\
    =&\sum_{i=0}^{2k-2} Q(x_0,\dots,x_{i},x_{i+3},\dots,x_{2k+1}) \otimes Q(x_i,x_{i+1},x_{i+2},x_{i+3}),\\
    S(x_0,&\dots,x_{2k+2})\\
    =&+\sum_{i=0}^{2k-1} S(x_0,\dots,x_{i},x_{i+3},\dots,x_{2k+2}) \otimes Q(x_i,x_{i+1},x_{i+2},x_{i+3})\\
    &+\sum_{i=0}^{2k-2} (-1)^i Q(x_0,\dots,x_{i},x_{i+4},\dots,x_{2k+2}) \otimes S(x_i,x_{i+1},x_{i+2},x_{i+3},x_{i+4}).
\end{align*}

It is easy to see that $Q(x_0,\dots,x_3)$ and $S(x_0,\dots,x_4)$ are both anti-symmetric in their arguments and vanish if any two arguments coincide.

\begin{proposition} The following statements hold:\nobreakpar
\begin{enumerate}
    \item $\mathcal{I}\bigl(Q(x_0,\dots,x_{2k+1})\bigr)= \textup{QLi}_{2k}(x_0,\dots,x_{2k+1})$ in $\mathcal{G}_k,$
    \item $S(x_0,\dots,x_{2k+2}\bigr)\equiv 0.$
\end{enumerate}
\end{proposition}
\begin{proof}
The first statement follows from the quadrangulation formula. To prove the second statement, notice that Proposition \ref{PropositionMainFunctionalEquation} implies that in $\mathcal{G}_k$
\[
\sum_{j=0}^{2k+2} (-1)^j \textup{QLi}_{2k}(x_0,\dots, \hat{x}_j,\dots,x_{2k+2})= \sum_{j=0}^{2k+2} (-1)^j \textup{QLi}_{2k}^{\textup{sym}}(x_0,\dots, \hat{x}_j,\dots,x_{2k+2})= 0.
\]
To prove $(2)$,  it is sufficient to show that 
\[
S(x_0,\dots,x_{2k+2})=\sum_{j=0}^{2k+2} (-1)^jQ(x_0,\dots,\hat{x}_j,\dots,x_{2k+2});
\]
we prove it by induction. For $k=1$ this is trivial. For $k\geq 2$, we have
\begin{subequations}
\begin{align}
    S(x_0,& \dots,x_{i},x_{i+3},\dots,x_{2k+2}) \nonumber\\
    =&+\sum_{j=0}^{i-1}(-1)^jQ(x_0,\dots,\hat{x}_j,\dots,x_{i},x_{i+3},\dots,x_{2k+2})\label{FormulaFirst1}\\
    &+\sum_{j=i+4}^{2k+1}(-1)^jQ(x_0,\dots,x_{i},x_{i+3},\dots,\hat{x}_j,\dots,x_{2k+2})\label{FormulaFirst2}\\
    &+(-1)^iQ(x_0,\dots,x_{i-1},x_{i+3},\dots,x_{2k+2})\label{FormulaFirst3}\\
    &+(-1)^{i+3}Q(x_0,\dots,x_{i},x_{i+4},\dots,x_{2k+2})\label{FormulaFirst4}.
\end{align}
\end{subequations}

Next,
\begin{subequations}
\begin{align}
Q(x_0,& \dots,x_{i},x_{i+4},\dots,x_{2k+2}) \otimes S(x_i,x_{i+1},x_{i+2},x_{i+3},x_{i+4}) \nonumber\\
=&+Q(x_0,\dots,x_{i},x_{i+4},\dots,x_{2k+2}) \otimes Q(x_{i+1},x_{i+2},x_{i+3},x_{i+4})\label{FormulaSecond1}\\
&-Q(x_0,\dots,x_{i},x_{i+4},\dots,x_{2k+2}) \otimes Q(x_i,x_{i+2},x_{i+3},x_{i+4})\label{FormulaSecond2}\\
&+Q(x_0,\dots,x_{i},x_{i+4},\dots,x_{2k+2}) \otimes Q(x_i,x_{i+1},x_{i+3},x_{i+4})\label{FormulaSecond3}\\
&-Q(x_0,\dots,x_{i},x_{i+4},\dots,x_{2k+2}) \otimes Q(x_i,x_{i+1},x_{i+2},x_{i+4})\label{FormulaSecond4}\\
&+Q(x_0,\dots,x_{i},x_{i+4},\dots,x_{2k+2}) \otimes Q(x_i,x_{i+1},x_{i+2},x_{i+3}).\label{FormulaSecond5}
\end{align}
\end{subequations}
In the formula
\begin{align*}
    S(x_0,&\dots,x_{2k+2})\\
    =&+\sum_{i=0}^{2k-1} S(x_0,\dots,x_{i},x_{i+3},\dots,x_{2k+2}) \otimes Q(x_i,x_{i+1},x_{i+2},x_{i+3})\\
    &+\sum_{i=0}^{2k-2} (-1)^iQ(x_0,\dots,x_{i},x_{i+4},\dots,x_{2k+2}) \otimes S(x_i,x_{i+1},x_{i+2},x_{i+3},x_{i+4})
\end{align*}
the terms coming from (\ref{FormulaFirst3}) and  (\ref{FormulaSecond5}) cancel each other; similarly for (\ref{FormulaFirst4}) and  (\ref{FormulaSecond1}).
It is easy to see that  every remaining term has exactly one index missing. We collect all terms not containing some index $j$ and get
\begin{subequations}
\begin{align}
&+\sum_{i>j}(-1)^jQ(x_0,\dots,\hat{x}_j,\dots,x_{i},x_{i+3},\dots,x_{2k+2})\otimes Q(x_i,x_{i+1},x_{i+2},x_{i+3})\label{FormulaThird1}\\
&+\sum_{i<j-3}(-1)^jQ(x_0,\dots,x_{i},x_{i+3},\dots,\hat{x}_j,\dots x_{2k+2}) \otimes Q(x_i,x_{i+1},x_{i+2},x_{i+3})\label{FormulaThird2}\\
&+(-1)^jQ(x_0,\dots,x_{j-1},x_{j+3},\dots,x_{2k+2}) \otimes Q(x_{j-1},x_{j+1},x_{j+2},x_{j+3})\label{FormulaThird3}\\
&+(-1)^j Q(x_0,\dots,x_{j-2},x_{j+2},\dots,x_{2k+2}) \otimes Q(x_{j-2},x_{j-1},x_{j+1},x_{j+2})\label{FormulaThird4}\\
&+(-1)^j Q(x_0,\dots,x_{j-3},x_{j+1},\dots,x_{2k+2}) \otimes Q(x_{j-3},x_{j-2},x_{j-1},x_{j+1})\label{FormulaThird5}.
\end{align}
\end{subequations}
In the formula above, terms (\ref{FormulaThird1}) come from (\ref{FormulaFirst1}), terms (\ref{FormulaThird2}) come from (\ref{FormulaFirst2}),  term (\ref{FormulaThird3}) comes from (\ref{FormulaSecond2}),  term (\ref{FormulaThird4}) comes from (\ref{FormulaSecond3}), and term (\ref{FormulaThird5}) comes from (\ref{FormulaSecond4}).

The sum of the terms (\ref{FormulaThird1})--(\ref{FormulaThird5}) equals to 
\[
(-1)^jQ(x_0,\dots,\hat{x}_j,\dots,x_{2k+1}),
\]
from where the statement follows.
\end{proof}

\subsection{Gangl formula in weight four}  \label{SectionHigherGanglFour}
As a first application, we prove Conjecture \ref{ConjectureHigherGangl} in weight four, known as the Gangl formula (\cite[Theorem 17]{Gan16}). Our proof simplifies that from \cite[\S 6]{GR18}. 

Zagier proved that 
$\Li_{2;1,1}^{\L}(a_1,a_2)+\Li_{2;1,1}^{\L}(1-a_1,a_2)$ and 
$\Li_{2;1,1}^{\L}(a_1,a_2)+\Li_{2;1,1}^{\L}\left(\frac{1}{a_1},a_2\right)$
have depth one. We suggest a geometric interpretation of his formulas. Consider six points $x_0,\dots,x_5\in \mathbb{P}^1_\F$ such that there exists a projective involution $\psi$ such that $\psi(x_0)=x_3,$  $\psi(x_1)=x_4$ and $\psi(x_2)=x_5.$
Then the following formula holds:
\begin{align*}
2\textup{QLi}_4^(x_0&,x_5,x_0,x_4,x_2,x_1)-2\textup{QLi}_4(x_5,x_4,x_0,x_2,x_3,x_4)\\
=&+\textup{QLi}_4(x_0,x_1,x_3,x_4) -\textup{QLi}_4(x_0,x_2,x_3,x_5)+\textup{QLi}_4(x_1,x_2,x_4,x_5)\\
&+2\textup{QLi}_4(x_0,x_2,x_1,x_5)+2\textup{QLi}_4(x_0,x_4,x_3,x_5)\\
&-2\textup{QLi}_4(x_1,x_2,x_4,x_3)-2\textup{QLi}_4(x_1,x_3,x_2,x_5)\\
&-2\textup{QLi}_4(x_2,x_3,x_4,x_5).
\end{align*}
Its proof is based on the following version of the Kummer equation: 
\[
\begin{split}
 \Li_{3}^{\L}&\left([x_0,x_1,x_3,x_4]\right)+ \Li_{3}^{\L}\left([x_0,x_2,x_3,x_5]\right)+ \Li_{3}^{\L}\left([x_1,x_2,x_4,x_5]\right)+2\Li_{3}^{\L}(1)\\
=&+2\Li_3^{\L}\left([x_0,x_1,x_3,x_5]\right)+2\Li_{3}^{\L}\left([x_0,x_2,x_4,x_5]\right)+ 2\Li_{3}^{\L}\left([x_1,x_2,x_4,x_3]\right)\\
&+2\Li_{3}^{\L}\left([x_0,x_1,x_2,x_4]\right)+2\Li_{3}^{\L}\left([x_0,x_1,x_3,x_2]\right)+2\Li_{3}^{\L}\left([x_2,x_0,x_5,x_1]\right).
\end{split}
\]
The first Zagier formula follows by expressing $\textup{QLi}_4$ via $\Li_{3,1}^{\L}$ using the  quadrangulation formula. The second Zagier formula can be easily derived from the first one.

\begin{proposition}[Gangl formula]
 For $a\in \F^{\times}$ and $x_0,\dots,x_4\in \mathbb{P}^1_\F$  the sum
\be \label{FormulaGanglWeightfour}
	\sum_{i=0}^4(-1)^i\textup{Li}^{\L}_{2;1,1}(a,[x_0,\dots,\widehat{x_i},\dots,x_4])
\ee
can be expressed via  classical polylogarithms.
\end{proposition}

\begin{proof}
We need to show that
$
\sum_{i=0}^4(-1)^i\textup{Li}_{2;1,1}^{\L}(a,[x_0,\dots,\widehat{x_i},\dots,x_4])\equiv 0.
$

We know that
 $S(x_0,\dots,x_{6})\equiv 0.$
More explicitly, we have 
\begin{align*}
& +S(x_0,x_3,x_4,x_5,x_6)\otimes Q(x_0,x_1,x_2,x_3)+S(x_0,x_1,x_4,x_5,x_6)\otimes Q(x_1,x_2,x_3,x_4)\\
 &+S(x_0,x_1,x_2,x_5,x_6)\otimes Q(x_2,x_3,x_4,x_5)+S(x_0,x_1,x_2,x_3,x_6)\otimes Q(x_3,x_4,x_5,x_6)\\
 &+S(x_0,x_4,x_5,x_6)\otimes S(x_0,x_1,x_2,x_3,x_4)-Q(x_0,x_1,x_5,x_6)\otimes S(x_1,x_2,x_3,x_4,x_5)\\
 &+Q(x_0,x_1,x_2,x_6)\otimes S(x_2,x_3,x_4,x_5,x_6)\equiv 0.
\end{align*}
Specializing to the divisor $x_6=x_4$ we get
\[
S(x_0,x_1,x_2,x_5,x_4)\otimes Q(x_2,x_3,x_4,x_5)\equiv Q(x_0,x_1,x_5,x_4)\otimes S(x_1,x_2,x_3,x_4,x_5).
\]
Element 
$S(x_0,x_1,x_2,x_5,x_4)\otimes Q(x_2,x_3,x_4,x_5)$
is invariant under transpositions $(24), (25)$ and $(15)$  but changes sign 
under transpositions $(01)$ and $(23).$ Thus
\[
S(x_0,x_1,x_2,x_5,x_4)\otimes Q(x_2,x_3,x_4,x_5)\equiv 0,
\]
which implies (\ref{FormulaGanglWeightfour}).
\end{proof}

\subsection{Higher Gangl formula in weight six} \label{SectionHigherGanglSix}
In this section, we prove Conjecture \ref{ConjectureHigherGangl} for $k=3.$ To show that Theorem \ref{TheoremDepth6} follows from it, we need to express 
\[
\Li_{3;1,1,1}^{\L}(a_1,a_2,a_3)+\Li_{3;1,1,1}^{\L}\left(\frac{1}{a_1},a_2,a_3\right).
\]
via functions (\ref{FormulaZagierWeightSix}) and polylogarithms of depth two. 
This was done by  Charlton, Gangl, and Radchenko (informal communication).
It remains to show that for any elements $Q_1, Q_2$ and $S$ we have  $Q_1\otimes Q_2 \otimes S\equiv 0.$  This would imply that $\mathcal{I}$ is well-defined, so $\Delta^{[2]}$ is bijective.

For this, we consider the following four degenerations of the equation  $S(x_0,\dots,x_{8})\equiv 0$. For clarity, we write $i$ instead of $x_i$ for arguments of $Q$ and~$S.$
\begin{align*}
0\equiv& D_1(x_0,\dots,x_6)\equiv -S(x_0,x_1,x_0,x_1,x_2,x_3,x_4,x_5,x_6)\\
\equiv &+Q(0,1,2,3)\otimes Q(0,1,3,4)\otimes S(0,1,4,5,6)\\
&-Q(0,1,2,3)\otimes S(0,1,3,4,5)\otimes Q(0,1,5,6)\\
&+Q(0,1,5,6)\otimes Q(0,1,4,5)\otimes S(0,1,2,3,4).\\
0\equiv & D_2(x_0,\dots,x_6)\equiv -S(x_0,x_1,x_0,x_2,x_0,x_3,x_4,x_5,x_6)\\
\equiv&+Q(0,1,5,6)\otimes S(0,1,2,4,5)\otimes Q(0,2,3,4) \\
&+Q(0,1,5,6)\otimes Q(0,1,2,5)\otimes S(0,2,3,4,5)\\
&-Q(0,2,3,4)\otimes Q(0,1,2,4)\otimes S(0,1,4,5,6).\\
\end{align*}
\begin{align*}
0\equiv& D_3(x_0,\dots,x_6)\equiv -S(x_0,x_1,x_0,x_2,x_3,x_0,x_4,x_5,x_6)\\
  \equiv&+Q(1,0,3,2)\otimes Q(1,0,4,3)\otimes S(1,0,4,6,5)-Q(0,4,2,3)\otimes Q(0,5,1,6)\otimes S(0,5,1,4,2)\\
 &-Q(0,5,1,6)\otimes Q(0,4,2,3)\otimes S(0,4,2,5,1)+Q(1,0,3,2)\otimes Q(1,0,5,6)\otimes S(1,0,5,3,4)\\
 &-Q(1,0,5,6)\otimes Q(1,0,3,2)\otimes S(1,0,3,5,4)-Q(3,0,1,2)\otimes Q(3,0,5,4)\otimes S(3,0,5,1,6)\\
 &+Q(3,0,1,2)\otimes Q(3,0,6,1)\otimes S(3,0,6,5,4)+Q(3,0,5,4)\otimes Q(3,0,1,2)\otimes S(3,0,1,5,6)\\
 &-Q(4,0,2,3)\otimes Q(4,0,1,2)\otimes S(4,0,1,6,5)+Q(5,0,1,6)\otimes Q(5,0,2,1)\otimes S(5,0,2,4,3)\\
 &+Q(5,0,1,6)\otimes Q(5,0,3,4)\otimes S(5,0,3,1,2)-Q(5,0,3,4)\otimes Q(5,0,1,6)\otimes S(5,0,1,3,2).
\end{align*}
\begin{align*}
0\equiv& D_4(x_0,\dots,x_6) \equiv -S(x_0,x_1,x_2,x_0,x_3,x_4,x_0,x_5,x_6)\\
 \equiv &+Q(0,3,1,2)\otimes Q(0,6,4,5)\otimes S(0,6,4,3,1)+Q(0,4,2,3)\otimes Q(0,5,1,6)\otimes S(0,5,1,4,2)\\
 &+Q(0,5,1,6)\otimes Q(0,4,2,3)\otimes S(0,4,2,5,1)+Q(0,5,3,4)\otimes Q(0,6,2,1)\otimes S(0,6,2,5,3)\\
 &+Q(0,6,2,1)\otimes Q(0,5,3,4)\otimes S(0,5,3,6,2)+Q(0,6,4,5)\otimes Q(0,3,1,2)\otimes S(0,3,1,6,4)\\
 &+Q(1,0,3,2)\otimes Q(1,0,5,6)\otimes S(1,0,5,3,4)-Q(1,0,5,6)\otimes Q(1,0,3,2)\otimes S(1,0,3,5,4)\\
 &+Q(1,0,5,6)\otimes Q(1,0,4,5)\otimes S(1,0,4,3,2)+Q(2,0,4,3)\otimes Q(2,0,5,4)\otimes S(2,0,5,6,1)\\
 &-Q(2,0,4,3)\otimes Q(2,0,6,1)\otimes S(2,0,6,4,5)+Q(2,0,6,1)\otimes Q(2,0,4,3)\otimes S(2,0,4,6,5)\\
 &-Q(3,0,1,2)\otimes Q(3,0,5,4)\otimes S(3,0,5,1,6)+Q(3,0,1,2)\otimes Q(3,0,6,1)\otimes S(3,0,6,5,4)\\
 &+Q(3,0,5,4)\otimes Q(3,0,1,2)\otimes S(3,0,1,5,6)+Q(4,0,2,3)\otimes Q(4,0,6,5)\otimes S(4,0,6,2,1)\\
 &-Q(4,0,6,5)\otimes Q(4,0,1,6)\otimes S(4,0,1,3,2)-Q(4,0,6,5)\otimes Q(4,0,2,3)\otimes S(4,0,2,6,1)\\
 &+Q(5,0,1,6)\otimes Q(5,0,3,4)\otimes S(5,0,3,1,2)-Q(5,0,3,4)\otimes Q(5,0,1,6)\otimes S(5,0,1,3,2)\\
 &-Q(5,0,3,4)\otimes Q(5,0,2,3)\otimes S(5,0,2,6,1)-Q(6,0,2,1)\otimes Q(6,0,3,2)\otimes S(6,0,3,5,4)\\
 &-Q(6,0,2,1)\otimes Q(6,0,4,5)\otimes S(6,0,4,2,3)+Q(6,0,4,5)\otimes Q(6,0,2,1)\otimes S(6,0,2,4,3).
\end{align*}

The key idea is to find equations where all terms involved are obtained from
\[
[[0,1,2,3,4,5,6]]:= Q(0,1,5,4)\otimes Q(0,1,5,6)\otimes S(0,1,6,2,3)
\]
by permutations of the points. The following relations hold:
\begin{align*}
0\equiv& D_5(x_0,\dots,x_6)\equiv D_2+D_3-D_1+(13)(46)D_1\\
\equiv
&+Q(0,1,5,6)\otimes S(0,1,2,3,5)\otimes Q(0,3,4,5)\\
&-Q(0,1,5,6)\otimes Q(0,1,4,5)\otimes S(0,1,2,3,4)\\
&-Q(0,3,4,5)\otimes Q(0,3,5,6)\otimes S(0,1,2,3,6).
\end{align*}
\begin{align*}
0\equiv& D_6\equiv D_2-(15)D_2\\
\equiv &+Q(0,5,1,6)\otimes Q(0,5,1,2)\otimes S(0,5,2,3,4)-Q(0,1,5,6)\otimes Q(0,1,5,2)\otimes S(0,1,2,3,4)\\
&-Q(0,4,2,3)\otimes Q(0,4,2,1)\otimes S(0,4,1,5,6)+Q(0,4,2,3)\otimes Q(0,4,2,5)\otimes S(0,4,5,1,6)\\
\equiv &-[[0,1,3,4,6,5,2]]+[[0,4,1,6,3,2,5]]-[[0,4,5,6,3,2,1]]+[[0,5,3,4,6,1,2]].
\end{align*}
\begin{align*}
0\equiv& D_7\equiv D_5+(05)D_5\\
\equiv &+Q(0,1,5,6)\otimes Q(0,1,5,4)\otimes S(0,1,4,2,3)+Q(0,3,5,4)\otimes Q(0,3,5,6)\otimes S(0,3,6,1,2)\\  
&-Q(1,5,0,6)\otimes Q(1,5,0,4)\otimes S(1,5,4,2,3)-Q(3,5,0,4)\otimes Q(3,5,0,6)\otimes S(3,5,6,1,2)\\
\equiv &[[0,1,2,3,6,5,4]]+[[0,3,1,2,4,5,6]]-[[1,5,2,3,6,0,4]]-[[3,5,1,2,4,0,6]].
\end{align*}
\begin{align*}
0\equiv& D_{8}\equiv +D_4-(15)(24)D_2-(26)(35)D_2-(13)(46)D_2-D_1\\
&\quad \quad \  -D_5+(13)(46)D_1+(123456)D_5-(12)(36)(45)D_1-(12)(36)(45)D_5\\
\equiv &+Q(0,1,3,2)\otimes Q(0,1,3,4)\otimes S(0,1,4,5,6)+Q(0,3,1,2)\otimes Q(0,3,1,6)\otimes S(0,3,6,4,5)\\
&+Q(0,4,6,5)\otimes Q(0,4,6,1)\otimes S(0,4,1,2,3)+Q(0,6,4,5)\otimes Q(0,6,4,3)\otimes S(0,6,3,1,2)\\
\equiv &[[0,1,5,6,2,3,4]]+[[0,3,4,5,2,1,6]]+[[0,4,2,3,5,6,1]]+[[0,6,1,2,5,4,3]].
\end{align*}

\begin{lemma} The following formula holds:
\[
[[\sigma(0),\sigma(1),\sigma(2),\sigma(3),\sigma(4),\sigma(5),\sigma(6)]]\equiv (-1)^{\sigma}[[0,1,2,3,4,5,6]].
\]
\end{lemma}
\begin{proof}
Since $D_6+(0,4)D_6\equiv 0$ we have
\be \label{FormulaProofSix1}
\begin{split}
0 \equiv &-[[0,1,3,4,6,5,2]]+[[0,5,3,4,6,1,2]]\\
&-[[4,1,3,0,6,5,2]]+[[4,5,3,0,6,1,2]].   
\end{split}
\ee
Applying to (\ref{FormulaProofSix1}) permutation $(01)(24)$ we get
\be \label{FormulaProofSix2}
\begin{split}
0\equiv&-[[1,0,3,2,6,5,4]]+[[1,5,3,2,6,0,4]]\\
&-[[2,0,3,1,6,5,4]]+[[2,5,3,1,6,0,4]].
\end{split}
\ee
Similarly, looking at  $D_7+(1,2)D_7\equiv 0$ we have
\be \label{FormulaProofSix3}
\begin{split}
0\equiv &+[[0,1,2,3,6,5,4]]-[[1,5,2,3,6,0,4]]\\
&+[[0,2,1,3,6,5,4]]-[[2,5,1,3,6,0,4]].
\end{split}
\ee
Next, adding (\ref{FormulaProofSix2}) and (\ref{FormulaProofSix3}) we get that 
\[
[[1,5,3,2,6,0,4]]\equiv -[[2,5,3,1,6,0,4]].
\]
Thus $[[0,1,2,3,4,5,6]]$ is anti-symmetric in $\{0,1,2,3\}.$

Now we look at $D_8+(04)D_8\equiv 0$:
\[
\begin{split}
0\equiv &-[[0,1,6,5,2,3,4]]+[[0,6,1,2,5,4,3]]\\
&-[[4,1,6,5,2,3,0]]+[[4,6,1,2,5,0,3]]
\end{split}
\]
or, after applying (23)(456),
\be \label{FormulaProofSix4}
\begin{split}
0\equiv &-[[0,1,4,6,3,2,5]]-[[0,1,4,3,6,5,2]]\\
&-[[5,1,4,6,3,2,0]]-[[5,1,4,3,6,0,2]]=0.
\end{split}
\ee
On the other hand, the following equality follows from $D_6+(05)D_6\equiv 0$:
\be \label{FormulaProofSix5}
\begin{split}
0\equiv&+[[0,1,4,3,6,5,2]]-[[0,1,4,6,3,2,5]]\\
&+[[5,1,4,3,6,0,2]]-[[5,1,4,6,3,2,0]].
\end{split}
\ee
Adding (\ref{FormulaProofSix4}) and (\ref{FormulaProofSix5}), 
we get 
\[
[[0,1,4,6,3,2,5]]+[[5,1,4,6,3,2,0]]\equiv 0
\]
and 
\[
[[0,1,4,3,6,5,2]]+[[5,1,4,3,6,0,2]]\equiv 0.
\]
Thus $[[0,1,2,3,4,5,6]]$ is anti-symmetric in $\{0,1,2,3,5,6\}.$
The statement of the lemma easily follows from $D_7\equiv 0.$
\end{proof}

Now we finish the proof.
From $D_1$, we have 
\begin{subequations}
\begin{align}
Q(0&,1,2,3)\otimes S(0,1,3,4,5)\otimes Q(0,1,5,6)\label{FGanglSix1}\\
\equiv&-Q(0,1,3,2)\otimes Q(0,1,3,4)\otimes S(0,1,4,5,6)\label{FGanglSix2}\\
&-Q(0,1,5,6)\otimes Q(0,1,5,4)\otimes S(0,1,4,2,3)\label{FGanglSix3}.
\end{align}
\end{subequations}
The permutation exchanging $(\ref{FGanglSix2})$ and $(\ref{FGanglSix3})$ is odd, so (\ref{FGanglSix1}) vanishes. The vanishing of (\ref{FGanglSix1}) implies that every term of the type $Q\otimes S \otimes Q$ vanishes.

Finally, consider degeneration
\begin{subequations}
\begin{align}
D_9(x_0,\dots,x_6)\equiv&S(x_0,x_1,x_0,x_2,x_3,x_2,x_4,x_5,x_6)\nonumber\\
 &+[2,0,3,1]\otimes[2,5,3,1,4]\otimes[2,5,0,6]\label{FormulaDegenerationNine1} \\ 
 &+[3,1,2,0]\otimes[3,5,2,0,6]\otimes[3,5,1,4]\label{FormulaDegenerationNine2} \\ 
 &-[2,0,3,1]\otimes[2,4,3,1]\otimes[2,4,0,6,5]\nonumber\\
 &+[3,1,2,0]\otimes[3,6,2,0]\otimes[3,6,1,5,4] \nonumber\\
 &-[5,2,0,6]\otimes[5,3,1,2,0]\otimes[5,3,1,4]\label{FormulaDegenerationNine5}.  
\end{align}
\end{subequations}
Terms (\ref{FormulaDegenerationNine1}), (\ref{FormulaDegenerationNine2}), and (\ref{FormulaDegenerationNine5}) are of the type $Q\otimes S \otimes Q,$ so vanish. We get that
\be
[2,0,3,1]\otimes[2,4,3,1]\otimes[2,4,0,6,5]\equiv [3,1,2,0]\otimes[3,6,2,0]\otimes[3,6,1,5,4].
\ee
It follows that the term $[3,1,4,2]\otimes[3,5,4,2]\otimes[3,5,1,7,6]$ changes sign under transpositions $(45).$ On the other hand, it is invariant under $(04),$ because in the Lie coalgebra, we have
\begin{align*}
[2,4,3,1]&\otimes[2,0,3,1]\otimes[2,4,0,6,5]\\
\equiv&-[2,0,3,1]\otimes[2,4,3,1]\otimes[2,4,0,6,5]-[2,0,3,1]\otimes[2,4,0,6,5]\otimes[2,4,3,1]\\
\equiv&-[2,0,3,1]\otimes[2,4,3,1]\otimes[2,4,0,6,5]\equiv[2,0,3,1]\otimes[2,4,3,1]\otimes[2,0,4,6,5].
\end{align*}
Since $(04)(45)(04)=(45)(04)(45),$ this implies that 
\[
[2,4,3,1]\otimes[2,0,3,1]\otimes[2,4,0,6,5]\equiv 0,
\]
so any term of the type $Q\otimes Q \otimes S$ vanishes. This finishes the proof of  Conjecture \ref{ConjectureHigherGangl} for $k=3$.

\bibliographystyle{alpha}      
\bibliography{Cluster_Polylogarithms_Bibliography}

\end{document}